\documentclass[12pt]{amsart}
\usepackage[pdftex,colorlinks,citecolor=blue,urlcolor=blue]{hyperref}
\usepackage{amssymb}
\usepackage{graphicx,color}
\usepackage{amsmath}
\usepackage{comment}
\usepackage{bbm}
\usepackage{MnSymbol}
\usepackage{cleveref}
\newtheorem{thmx}{Theorem}

\newtheorem*{mainconj}{Nonnegativity Conjecture for $\gamma$-vectors of symmetric edge polytopes}

\newtheorem{theorem}{Theorem}[section]
\newtheorem{proposition}[theorem]{Proposition}
\newtheorem{lemma}[theorem]{Lemma}
\newtheorem{corollary}[theorem]{Corollary}
\theoremstyle{definition}
\newtheorem{definition}[theorem]{Definition}

\newtheorem{example}[theorem]{Example}

\newtheorem{conjecture}[theorem]{Conjecture}
\newtheorem{remark}[theorem]{Remark}

\DeclareMathOperator{\lk}{\mathrm{lk}}

\DeclareMathOperator{\Var}{\mathrm{Var}}
\DeclareMathOperator{\cy}{\mathrm{cy}}

\newcommand{\N}{\mathbb{N}} 
\newcommand{\Z}{\mathbb{Z}} 

\newcommand{\R}{\mathbb{R}}
\newcommand{\cH}{\mathcal{H}}
\newcommand{\conv}{\mathop{\rm conv}\nolimits}
\newcommand{\Pc}{\mathcal{P}}
\newcommand{\Pb}{\mathbb{P}}
\newcommand{\Eb}{\mathbb{E}}

\title{On the gamma-vector of symmetric edge polytopes}
\author{Alessio D'Al\`{i}}
\address{Department of Mathematics, University of Osnabr\"{uc}k, Germany}
\curraddr{Dipartimento di Matematica, Politecnico di Milano, Italy}
\email{alessio.dali@polimi.it}
\author{Martina Juhnke-Kubitzke, Daniel K\"ohne}
\address{Department of Mathematics, University of Osnabr\"{uc}k, Germany}
\email{juhnke-kubitzke@uni-osnabrueck.de}
\email{daniel.koehne@uni-osnabrueck.de}
\author{Lorenzo Venturello}
\address{Department of Mathematics, KTH Royal Institute of Technology, Stockholm, Sweden and Department of Mathematics, University of Pisa, Pisa, Italy}
\email{lven@kth.se lorenzo.venturello@unipi.it}


\begin{document}

\maketitle
\begin{abstract}
  We study $\gamma$-vectors associated with $h^*$-vectors of symmetric edge polytopes both from a deterministic and a probabilistic point of view. On the deterministic side, we prove nonnegativity of $\gamma_2$ for any graph and completely characterize the case when $\gamma_2 = 0$. The latter also confirms a conjecture by Lutz and Nevo in the realm of symmetric edge polytopes. On the probabilistic side, we show that the $\gamma$-vectors of symmetric edge polytopes of most Erd\H{o}s-R{\'e}nyi random graphs are asymptotically almost surely nonnegative up to any fixed entry. This proves that Gal's conjecture holds asymptotically almost surely for arbitrary unimodular triangulations in this setting. 
\end{abstract}
\section{Introduction}
\emph{Symmetric edge polytopes} are a class of lattice polytopes that has seen a surge of interest in recent years for their intrinsic combinatorial and geometric properties \cite{FirstSEP, HKM, OT-2021, OT-2020, CDK}  as well as for their relations to 
metric space theory \cite{Ver, GP, DeHo}, optimal transport \cite{WassDist} and physics, where they appear in the context of the Kuramoto synchronization model \cite{CDM, Chen} (see \cite{DDM} for a more detailed account of these connections).

Given a finite simple graph $G=([n],E)$, the associated \emph{symmetric edge polytope} $\Pc_G$ is defined as
\begin{equation*}
    \Pc_G=\conv(\pm(e_i-e_j)~:~ij\in E).
\end{equation*}
On the one hand, the dependence on a graph allows for graph-theoretical characterizations of some polytopal properties: for instance, Higashitani proved in \cite[Corollary 2.3]{Hig} that a symmetric edge polytope $\mathcal{P}_G$ arising from a connected graph $G$ is simplicial if and only if $G$ contains no even cycles, and this is also equivalent to $\mathcal{P}_G$ being smooth. While for simpliciality there hence exist infinitely many symmetric edge polytopes with this property, the behavior is different for simplicity. In this case, there is only a finite list of graphs whose symmetric edge polytopes are simple (see \Cref{prop:simple}). On the other hand, there are several pleasant properties that are shared by any symmetric edge polytope, independent of the underlying graph: all of these polytopes are known to admit a pulling regular unimodular triangulation \cite{OH14, HJM} and to be centrally symmetric, terminal and reflexive \cite{Hig}.  In particular, by this latter property, it follows from work of Hibi \cite{Hib} that their $h^*$-vectors are palindromic. Thus, given the $h^\ast$-vector $h^\ast(\Pc_G)=(h^\ast_0,\ldots,h^\ast_d)$ of a symmetric edge polytope, one can define the $\gamma$-vector of $\Pc_G$ by applying the following change of basis:
\begin{equation} \label{eq:gamma_intro}
    \sum_{i=0}^{\lfloor \frac{d}{2}\rfloor} \gamma_i t^i(t+1)^{d-2i}=
\sum_{j=0}^{d} h^\ast_j t^{j}.
\end{equation}
Obviously, $\gamma(\Pc_G)=(\gamma_0,\ldots,\gamma_{\lfloor\frac{d}{2}\rfloor})$ stores the same information as $h^\ast(\Pc_G)$ in a more compact form. More generally, in the same way, one can associate a $\gamma$-vector with any symmetric vector and this has been done and studied extensively in a lot of cases. One of the most prominent examples in topological combinatorics, which is strongly related to the just mentioned example of $h^*$-vectors of reflexive polytopes, are $h$-vectors of simplicial spheres. For \emph{flag} spheres, Gal's conjecture \cite{Gal} states that their $\gamma$-vectors are nonnegative. Several related conjectures exist, including the Charney--Davis conjecture \cite{CD}, claiming nonnegativity only for the last entry of the $\gamma$-vector, and the Nevo--Petersen conjecture \cite{NP} which even conjectures the $\gamma$-vector of a flag sphere to be the $f$-vector of a balanced simplicial complex. Those conjectures have been a very active area of research in the last few years. However, even though they could be solved in special cases \cite{Aisbett,A2012,AV,DavisOkun,Gal,LaN,NP,NPT} and new approaches have been developed towards their solution \cite{CN,CN-2021}, they remain wide open in general. For an introduction to $\gamma$-nonnegativity, we invite the interested reader to consult the surveys \cite{Ath} and \cite[Sections 3 and 6]{Bra}.

If a polytope $\Pc$ admits a regular unimodular triangulation $\Delta$, which is the case for symmetric edge polytopes, then the restriction of $\Delta$ yields a unimodular triangulation of the boundary complex of $\Pc$, as well. If, in addition, $\Pc$ is reflexive, it is well-known that the $h^\ast$-vector of $\Pc$ equals the $h$-vector of any unimodular triangulation $\Delta$ of its boundary, which in particular is a simplicial sphere. This provides a link between the study of the $\gamma$-vector of $\mathcal{P}_G$ and the rich world of conjectures on the $\gamma$-nonnegativity of simplicial spheres; however, note that the objects we are interested in will \emph{not} be flag in general. Despite the lack of flagness, in all the cases known so far the $\gamma$-vector of $\mathcal{P}_G$ is nonnegative, and this brought Ohsugi and Tsuchiya to formulate the following conjecture, which is the starting point of this paper:
\begin{mainconj}\cite[Conjecture 5.11]{OT-2021}\label{MainConjecture}
Let $G$ be a graph. Then $\gamma_i(\mathcal{P}_G) \geq 0$ for every $i \geq 0$.
\end{mainconj}
On the one hand, it is already known and follows e.g.~from \cite{BR} that a weaker property, namely, unimodality of the $h^\ast$-vector holds. On the other hand, though it is tempting to hope that even the stronger property of the $h^\ast$-polynomial being real-rooted is true, this is not the case in general, as shown by the $5$-cycle. 
The $\gamma$-nonnegativity conjecture above has been verified for special classes of graphs, mostly by direct computation: as shown in \cite[Section 5.3]{OT-2021}, such classes encompass cycles, suspensions of graphs (which include both complete graphs and wheels), outerplanar bipartite graphs and complete bipartite graphs. This last instance was originally proved in \cite{HJM} but was generalized in \cite{OT-2021} to bipartite graphs $\widetilde{H}$ obtained from another bipartite graph $H$ as in \cite[p. 708]{OT-2021}.
The main goal of this paper is to provide some supporting evidence to the $\gamma$-nonnegativity conjecture, independent of the graph. We take two different approaches: a \emph{deterministic} and a \emph{probabilistic} one.

In the deterministic part, developed in Sections \ref{sec:nonnegativity} and \ref{sec:lutznevo}, we focus on the coefficient $\gamma_2$. Through some delicate combinatorial analysis, we are able to prove that $\gamma_2$ is always nonnegative. Moreover, we provide a characterization of those graphs for which $\gamma_2(\Pc_G) = 0$:
\begin{thmx}[Theorems \ref{thm: gamma2 >0} and \ref{cor: gamma_2=0}]\label{thm:A}
Let $G = ([n],E)$ be a graph. Then $\gamma_2(\Pc_G) \geq 0$. Moreover, if $G$ is 2-connected, then $\gamma_2(\Pc_G)=0$ if and only if either $n<5$, or $n\geq 5$ and $G$ is isomorphic to one of the following two graphs:
\begin{itemize}
    \item the graph $G_n$ with edge set $\{12\}\cup\{1k,2k~:~ k\in \{3,\dots,n\}\}$; or
    \item the complete bipartite graph $K_{2,n-2}$.
\end{itemize}
\end{thmx}
The ``if" part of the equality statement can be deduced from the results in \cite{HJM} and \cite{OT-2021}, where the authors compute explicitly the $\gamma$-vector of the families of graphs appearing in \Cref{thm:A}. See the proof of \Cref{cor: gamma_2=0} for more details.
As an application, we confirm a conjecture by Lutz and Nevo \cite[Conjecture 6.1]{LN}, characterizing flag spheres with $\gamma_2=0$, in the restricted context of some natural Gr\"obner-induced triangulations of boundaries of symmetric edge polytopes: see Theorem \ref{thm:lutznevoSEP}. We want to point out that the symmetric edge polytopes of the graphs from the second part of \Cref{thm:A} indeed admit a flag triangulation. 

Finally, Section \ref{sec:probabilistic} brings random graphs into the picture. The Erd\H{o}s-R{\'e}nyi model $G(n, p(n))$ is one of the most popular and well-studied ways to generate a graph on the vertex set $[n]$ via a random process: for a graph $G\in G(n,p)$, the probability of $ij$ with $1\leq i<j\leq n$ being an edge of $G$ equals $p(n)$, and all of these events are mutually independent. Our question is then: for an Erd\H{o}s-R{\'e}nyi graph $G\in G(n,p)$, how likely is it that the entries of the $\gamma$-vector of $\mathcal{P}_G$ are nonnegative?  As an extension, we pose the question of how \emph{big} those entries will most likely be. Our main result, answering both questions, is the following:
\begin{thmx}[Theorems \ref{gamma:no cycles} and \ref{thm:gammaconcentration1}]\label{thm:B}
 Let $k$ be a positive integer. For the Erd\H{o}s-R{\'e}nyi model $G(n,p(n))$, where $p(n)=n^{-\beta}$ for some $\beta>0$, $\beta \neq 1$, the following statements hold: 
\begin{itemize}
    \item \emph{(subcritical regime)} if $\beta > 1$, then asymptotically almost surely $\gamma_\ell = 0$ for all $\ell\geq 1$;
    \item \emph{(supercritical regime)} if $0 < \beta < 1$, then asymptotically almost surely $\gamma_{\ell} \in \Theta(n^{(2-\beta)\ell})$ for every $0 < \ell \leq k$.
\end{itemize}
\end{thmx}
In particular, this shows that $\gamma_\ell\geq 0$ for $1\leq \ell\leq k$ with high probability, thereby proving that (up to a fixed entry of the $\gamma$-vector) Gal's conjecture holds with high probability. To prove this result, we need to distinguish two regimes: subcritical ($\beta>1$) and supercritical ($0<\beta<1)$, the subcritical one being the easier one. Along the proof, we derive concentration inequalities for the number of non-faces and faces of the triangulation of $\Pc_G$ studied in \cite[Proposition 3.8]{HJM}. 

The paper is structured as follows. In \Cref{sec:Prel} we provide necessary background on graphs, simplicial complexes and polytopes and prove some basic statements for symmetric edge polytopes, including the characterization of their edges and of when they are simple. \Cref{sec:nonnegativity} is devoted to the proof of \Cref{thm:A}, whereas in \Cref{sec:lutznevo} we prove the previously mentioned conjecture by Lutz and Nevo in our setting. Finally, \Cref{sec:probabilistic} contains the proof of \Cref{thm:B}.

\subsection*{Acknowledgements}
We wish to thank Emanuele Delucchi and Hidefumi Ohsugi for their useful comments.
L.~V. is funded by the G\"oran Gustafsson foundation.

\section{Preliminaries}\label{sec:Prel}
In this section, we provide the relevant definitions and notations concerning graph theory, simplicial complexes, polytope theory and in particular symmetric edge polytopes. (For more background, we refer to \cite{Diestel}, \cite{Stanley-greenBook} and \cite{BrGu}.)

\subsection{Graphs}
Let $G=(V(G),E(G))$ be a simple graph with vertex set $V(G)$ and edge set $E(G)$. For $\{v, w\}\in E(G)$, we use the shorthand notation $vw$. Though, a priori, all graphs in this paper are undirected, we often consider different orientations of the edges. We then write $v\to w$ and $w\to v$ for the directed edges going from $v$ to $w$ and $w$ to $v$, respectively.  A \emph{path} of length $n$, or \emph{$n$-path}, denoted by $P_n$, is the graph $P_n=(\{v_0,\ldots,v_n\},\{v_0v_1,v_1v_2,\ldots,v_{n-1}v_{n}\})$. The graph $C_n=(\{v_1,\ldots,v_n\},\{v_1v_2,\ldots,v_{n-1}v_{n},v_nv_1\})$ is called \emph{$n$-cycle}.  The \emph{distance} $d(u,v)$ between vertices $u$ and $v$ is given as the length of a shortest path from $u$ to $v$. We denote by $G\cup e$ and $G\setminus e$ the graph obtained from $G$ by \emph{adding} and \emph{removing} an edge $e$, respectively. We define \emph{deleting} a vertex set $V'\subseteq V$ as $G\setminus V'=(V\setminus V',E\setminus \{e\ :\ e\cap V'\neq \emptyset\})$. The \emph{cyclomatic number} of a graph $G$ with $c$ connected components is defined as $\cy(G)=|E|-|V|+c$. 
It is well-known that $G$ is $2$-connected if it has an \emph{open ear decomposition}, meaning that $G$ is either a cycle (the closed ear) or can be obtained from such by successively attaching paths (the open ears) whose internal vertices are disjoint from the previous ears and whose two end vertices belong to an earlier ear. It is easy to see that the number of ears in any such decomposition equals the cyclomatic number of $G$. Moreover, every graph decomposes uniquely into its \emph{$2$-connected components}, i.e., inclusion-maximal $2$-connected subgraphs, where when speaking about this decomposition we also consider  single edges to be $2$-connected. We further use $K_n$ and $K_{n,m}$ to denote respectively the complete graph on $n$ vertices and the complete bipartite graph on $n$ and $m$ vertices. If two graphs $G$ and $H$ both contain a subgraph isomorphic to a $k$-clique $K_k$, the graph $G\oplus_k H$ obtained by gluing $G$ and $H$ together along $K_k$ is called the \emph{$k$-clique sum} of $G$ and $H$.

\subsection{Simplicial complexes}
A \emph{simplicial complex} $\Delta$ on vertex set $V$  is
any collection of subsets of $V$ closed under inclusion. The elements of $\Delta$ are called \emph{faces}, and a face that is maximal with respect to inclusion is called \emph{facet}. We will sometimes write $\langle F_1, ..., F_m\rangle$ to denote the simplicial complex with facets $F_1, ..., F_m$. A set which is not in $\Delta$ is a \emph{non-face}. A non-face is called \emph{minimal} if it is minimal with respect to inclusion.  The \emph{dimension} of a face $F$ is defined as $\dim(F):=|F|-1,$ and the dimension of $\Delta$ is $\dim(\Delta)=\max(\dim(F)\ :\ F\in\Delta).$  0-dimensional and 1-dimensional faces of $\Delta$ are called \emph{vertices} and \emph{edges}, respectively. The $1$-skeleton of $\Delta$ is the simplicial complex consisting of all edges and vertices of $\Delta$. For $v\in V$, its \emph{degree}  $d(v)$ is the degree of $v$ in the $1$-skeleton of $\Delta$ (as a graph). Given a $(d-1)$-dimensional simplicial complex $\Delta$, its $f$-vector $f(\Delta)=(f_{-1}(\Delta),f_0(\Delta),\ldots,f_{d-1}(\Delta))$ is defined by $f_i(\Delta)=|\{f\in\Delta\ :\ \dim(F)=i\}|$ for $-1\le i\le d-1$ and its \emph{$h$-vector} $h(\Delta)=(h_0(\Delta),h_1(\Delta),\ldots,h_d(\Delta))$ by 
\begin{equation}\label{eqn: h-vector}
   h_j(\Delta)=\sum_{i=0}^j(-1)^{j-1}\binom{d-i}{d-j}f_{i-j}(\Delta)
\end{equation}
for $0\le i\le d.$  Then, $f(\Delta,x)=\sum_{i=-1}^{d-1}f_i(\Delta)x^i$ and $h(\Delta,x)=\sum_{i=0}^{d}h_i(\Delta)x^i$ are called the  \emph{$f$}- and \emph{$h$-polynomial} of $\Delta$, respectively. If $h(\Delta)$ is \emph{symmetric}, i.e., $h_i(\Delta)=h_{d-i}(\Delta)$ for every choice of $i$, the $\gamma$-vector $\gamma(\Delta)=(\gamma_0(\Delta),\gamma_1(\Delta), \ldots,\gamma_{\lfloor\frac{d}{2}\rfloor}(\Delta))$ of $\Delta$ is defined via 
\begin{equation*}
    h(\Delta,x)=\sum_{i=0}^{\lfloor\frac{d}{2}\rfloor}\gamma_i(\Delta)x^i(1+x)^{d-2i}.
\end{equation*}
One way to study a simplicial complex $\Delta$ locally is to look at the \emph{link} of a face $F$, i.e.~the subcomplex of $\Delta$ defined as $\lk_{\Delta}(F)=\{H\in \Delta\ :\ H\cap F=\emptyset,\ F\cup H\in\Delta\}$.
\\
For an edge $F=\{v,w\}\in\Delta$, the \emph{edge contraction} $\Delta/F$ of $\Delta$ at $F$ is the simplicial complex obtained from $\Delta$ by identifying $v$ with $w$ in all faces of $\Delta$, i.e.,
    \[
        \Delta/F = \{H \in \Delta~:~ v\notin H\}\cup \{H\setminus\{v\} \cup \{w\}~:~ v\in H\}.
    \]
The \emph{face deletion} $\Delta\setminus F$ of a face $F\in \Delta$ from $\Delta$ is defined as $\Delta\setminus F = \{H\in \Delta~: ~F\not \subseteq H\}$.
Given simplicial complexes $\Delta$ and $\Gamma$, the simplicial complex $\Delta\star\Gamma=\{F\cup H\ :\ F\in\Delta,\ H\in\Gamma  \}$ is called the \emph{join} of $\Delta$ and $\Gamma.$

\subsection{Lattice polytopes}
A \emph{lattice polytope} $\Pc$ is the convex hull of finitely many points of a lattice in $\R^d$, typically $\Z^d$. An example, which is of importance for us, is provided by the \emph{$d$-dimensional cross-polytope} $\Diamond_d=\conv(\pm e_i\ :\ 1\le i\le d)$.  
A \emph{triangulation} $\mathcal{T}$ of a $d$-dimensional lattice polytope $\Pc$ is a subdivision into simplices of dimension at most $d$. Such a triangulation is \emph{unimodular} if all its simplices are, i.e., they have normalized volume $1$. The triangulation $\mathcal{T}$ is called \emph{flag} if its minimal non-faces all have cardinality $2$. 

Ehrhart \cite{Ehrhart} proved that $|n\Pc\cap\Z^d|$, i.e.~the number of lattice points in the $n$-th dilation of $\Pc$, is given by a polynomial $E(\Pc,n)$ of degree $\dim P$ in $n$ for all integers $n\geq 0$. The  \emph{$h^\ast$-polynomial} $h^\ast(\Pc,x)=h_0^\ast(\Pc)+h_1^\ast (\Pc)x+\cdots +h_d^\ast(\Pc) x^d$ of a $d$-dimensional lattice polytope $\Pc$ is obtained by applying a particular change of basis to $E(\Pc,n)$; namely, 
\[
E(\Pc,n)=h_0 ^\ast(\Pc) {n+d\choose d}+h_1 ^\ast(\Pc) {n+d-1\choose d} +\cdots + h_d^\ast (\Pc){n\choose d} \, .
\]
Stanley \cite{Stanley} showed that $h^\ast(\Pc,x)$ has only nonnegative coefficients. By work of Hibi \cite{Hib}, $h^\ast(\Pc,x)$ is \emph{palindromic}, i.e., $h^\ast(\Pc,x)=x^d h^\ast\left(\Pc,\frac{1}{x}\right)$, if and only if $\Pc$ is reflexive. If $\Pc$ has a unimodular triangulation $\mathcal{T}$, then $h^\ast(\Pc,x)$  coincides with the $h$-polynomial of $\mathcal{T}$. If, in addition, $\Pc$ is reflexive, then $h^\ast(\Pc,x)$ is also equal to the $h$-polynomial of the unimodular triangulation of the boundary of $\Pc$ induced by $\mathcal{T}$.

\subsection{Basic properties of symmetric edge polytopes}
Let $G=([n],E)$ be a simple graph with \emph{symmetric edge polytope} $\Pc_G$. For $ij\in E$, we will call the vertices $e_i-e_j$ and $e_j-e_i$ of $\Pc_G$ an \emph{antipodal pair}.  In the following, we will always identify a vertex $e_i-e_j$ of $\Pc_G$ with the oriented edge $i\to j$ and use the short-hand notation $e_{i,j}=e_i-e_j$. Though $G$ is unoriented, we can naturally orient each cycle of $G$ by orienting its edges either clockwise or counter-clockwise. By abuse of notation, we refer to those cycles as the \emph{oriented} cycles of $G$.

Turning to triangulations of symmetric edge polytopes, we recall that it was shown in \cite{OH14} that $\Pc_G$ admits a regular unimodular triangulation. It is well-known (see e.g., \cite[Corollary 8.9]{Sturmfels}) that such a triangulation can be obtained from the Gr\"obner basis of the toric ideal of $\Pc_G$ (with respect to the degrevlex order), provided in \cite[Proposition 3.8]{HJM}, as follows:

\begin{lemma}\label{lem: nonfaces of Delta<}
	Let $<$ be a total order on the edges $E$ of $G$. Then there exists a unimodular triangulation $\Delta_<$ of $\partial \mathcal{P}_G$ such that the $F$ is a non-face of $\Delta_<$ if and only if it contains at least one subset of the following form:
	\begin{itemize}
	\item[(i)] an antipodal pair; 
		\item[(ii)] an $\ell$-element subset of an oriented $(2\ell-1)$-cycle of $G$; or,
 		\item[(iii)] an $\ell$-element subset of an oriented $2\ell$-cycle of $G$ not containing its $<$-minimal edge.
	\end{itemize}
\end{lemma}
We want to remark that the triangulation $\Delta_<$ extends to a regular unimodular triangulation of $\Pc_G$ by coning over the origin. In (iii) an oriented edge $i\to j$ of an oriented cycle $C$ is called \emph{$<$-minimal} if $ij$ is minimal with respect to $<$ among $\{k\ell~:~k\to\ell \in E(C)\}$. It is apparent that the triangulation of \Cref{lem: nonfaces of Delta<} depends on the chosen ordering $<$. However, any edge of $\Pc_G$ is necessarily a face of any such triangulation. Complementing \cite[Theorem 3.1]{HJM} which characterizes facets of symmetric edge polytopes, we provide the following characterization of their edges:

\begin{theorem}\label{Conjecture:edges}
Let $G=([n],E)$ be a graph. 
    Two oriented edges of $G$ form an edge of $\mathcal{P}_G$ if and only if they are not contained in a directed $3$- or $4$-cycle of $G$.
\end{theorem}
\begin{proof}
    The ``only if''-part directly follows from \Cref{lem: nonfaces of Delta<} and the paragraph preceding this theorem. For the reverse statement, let $i\to j, k\to \ell$ oriented edges of $G$ neither lying in a directed $3$- nor in a directed $4$-cycle of $G$. The aim is to construct a supporting hyperplane of $\Pc_G$ only containing the vertices $e_{i,j}$ and $e_{k,\ell}$. More precisely, we construct $a\in \R^n$ such that $a^Te_{i,j}=a^Te_{k,\ell}>a^Ty$ for every vertex $y$ of $\Pc_G$ different from $e_{i,j}$ and $e_{k,\ell}$. We distinguish different cases.

{\sf Case 1:} $\{i,j\}\cap\{k,\ell\}=\emptyset$. Since $i\to j$ and $k\to\ell$ do not lie in an oriented $4$-cycle, we have $|\{i\ell,jk\}\cap E|\leq 1$. If $|\{i\ell,jk\}\cap E|=0$ then it is easy to verify that setting  $a_i=a_k=1$, $a_j=a_\ell=-1$ and $a_m=0$, otherwise, works. If $|\{i\ell,jk\}\cap E|= 1$, then without loss of generality assume $i\ell\in E$. In this case setting $a_i=1$, $a_j=-2$, $a_k=2$, $a_\ell=-1$ and $a_m=0$, otherwise, has the required properties. 

{\sf Case 2:} $\{i,j\}\cap\{k,\ell\}\neq\emptyset$. First assume $i=k$. In this case, we set $a_i=1$, $a_j=a_\ell=-1$ and $a_m=0$, otherwise. Similarly, if $j=\ell$, setting $a_j=-1$, $a_i=a_k=1$ and $a_m=0$, otherwise, works. Finally assume that $i=\ell$ or $j=k$. By symmetry, we only need to consider the case $i=\ell$. Since $i\to j$ and $k\to i$ do not lie in a directed $3$-cycle, it follows that $jk\notin E$. Similarly, as $i\to j$ and $k\to i$ do not lie in a directed $4$-cycle, the vertices $j$ and $k$ do not have common neighbors other than $i$. We can then set $a_j=-2$, $a_k=2$, $a_i=0$, $a_p=-1$ if $jp \in E$, $a_q=1$ if $kq\in E$ and $a_m=0$, otherwise and this is well-defined by the previous arguments. It is again easy to see that this choice of $a$ works.
\end{proof}

\Cref{Conjecture:edges} allows us to characterize simple symmetric edge polytopes in terms of their graphs. 

\begin{proposition}\label{prop:simple}
Let $G=([n],E)$ be a graph with $E\neq \emptyset$. Then $\Pc_G$ is simple if and only if, after removing isolated vertices, $G\in\{P_1,2P_1 ,P_2,C_3,C_4\}$.
\end{proposition}
\begin{proof}
    It follows by direct computation that  $\Pc_G$ is a $1$-simplex, a $4$-gon, a $4$-gon, a $6$-gon and a $3$-cube if $G$ is equal to $P_1$, $2P_1$, $P_2$, $C_3$ and $C_4$, respectively.\\
    Assume that $G$ is connected and fix an oriented edge $i \to j$. Observe that if there exists an edge $k\ell$ such that both $k\to\ell$ and $\ell\to k$ lie in a $3$- or $4$-cycle together with $i\to j$, then the subgraph of $G$ induced by the vertices $i,j,k,\ell$ (which must be all distinct) is isomorphic to $K_4$. Consider then the connected subgraph $H$ of $G$ obtained by removing all edges $st$ such that the induced subgraph of $G$ on the vertex set $\{i,j,s,t\}$ is isomorphic to $K_4$. 
    By construction, for every edge $k\ell$ of $H$ different from $ij$, at least one of the oriented edges $k\to \ell$ and $\ell\to k$ does not lie in a $3$- or $4$-cycle with $i \to j$. By \Cref{Conjecture:edges}, this implies that at least one of the vertices $e_{k,\ell}$ or $e_{\ell,k}$ is adjacent to $e_{i,j}$ both in $\mathcal{P}_H$ and in $\mathcal{P}_G$. As all edges of $\mathcal{P}_H$ containing the vertex $e_{i,j}$ are also edges of $\mathcal{P}_G$ (and vice versa), we conclude that the number of edges containing $e_{i,j}$ in $\mathcal{P}_G$ is greater than or equal to $|E(H)|-1$. Hence, $\mathcal{P}_G$ cannot be simple if $|E(H)|>\dim(\mathcal{P}_G)+1=n$. In order for $\Pc_G$ to be simple, we must hence have $|E(H)|\in\{n-1,n\}$. If $|E(H)|=n-1$, then $H$ is a tree, while if $|E(H)|=n$, then $H$ can be built starting from a cycle and taking successive $1$-clique sums with single edges. Both cases can only happen if $G = H$, since otherwise $H$ would contain at least two distinct $3$-cycles. Using \cite[Proposition 4.2]{OT-2021}, it follows that if $G$ is connected and $\mathcal{P}_G$ is simple, then $\mathcal{P}_G$ is the free sum of the symmetric edge polytope of a cycle and some segments (since the symmetric edge polytope of $P_1$ is a segment). Since the free sum of two polytopes of dimension greater than zero is simple if and only if the polytopes are segments, we are left with the following possibilities: either $G\in\{P_1,P_2\}$ or $G\cong C_k$, for some $k\geq 3$. For analogous reasons, if $G$ is not connected and $\Pc_G$ is simple, $\Pc_G$ must be the free sum of two segments, i.e., $G\cong 2P_1$, the disjoint union of two edges. Finally, assume that $G\cong C_k$, for some $k\geq 5$. Applying again \Cref{Conjecture:edges}, we conclude that the number of edges of $\mathcal{P}_G$ containing $e_{i,j}$ equals $2(k-1)>k-1=\dim(\mathcal{P}_G)$ and hence $\Pc_G$ is not simple. 
\end{proof}
\section{\texorpdfstring{Nonnegativity of $\gamma_2$}{Nonnegativity of gamma\_2}} \label{sec:nonnegativity}
The aim of this section is to prove \Cref{thm:A}; namely, to show that $\gamma_2(\Pc_G)$ is nonnegative for any graph $G$, and to characterize which graphs attain the equality $\gamma_2(\Pc_G)=0$.
In \cite[Corollary 3.1]{OH14} the authors prove that, when $G$ is connected, $\Pc_G$ has a unimodular triangulation, and this was made more explicit in \cite{HJM} by providing a Gr\"obner basis. This triangulation depends on an order $<$ on the set of edges $E$, and, in particular,  different orders might yield non-isomorphic simplicial complexes. However, all of them are cones over the corresponding triangulation $\Delta_<$ of the boundary of $\Pc_G$ (see \Cref{lem: nonfaces of Delta<}). As the $h^\ast$-vector of a lattice polytope which admits a unimodular triangulation is equal to the $h$-vector of such a triangulation, we can write $\gamma_2(\Pc_G)$ in terms of the number of vertices and edges of $\Delta_<$, and these numbers do not depend on the order $<$. 
Our first goal is to write the number $\gamma_2(\Pc_G)$ as a function of certain invariants of the graph. For this aim, given a graph $G=([n],E)$ and a fixed total order $<$ on $E$, let
\[
    n_1(G):=\binom{2|E|}{2}-|E|-f_1(\Delta_<).
\]

In other words, $n_1(G)$ equals the number of edges of the $(|E|-1)$-dimensional cross-polytope on vertex set $\{e_{i,j},e_{j,i} ~:~ ij\in E\}$ that are non-edges of $\Delta_{<}$. By \Cref{lem: nonfaces of Delta<}, $n_1(G)$ is equal to the number of pairs of oriented edges of $G$ where the two unoriented edges are different and the pair satisfies at least one of the following:
\begin{itemize}
    \item[(i)] the pair is contained in an oriented $3$-cycle;
    \item[(ii)] the pair is contained in an oriented $4$-cycle, and none of its edges is the $<$-minimal edge of such a cycle.
\end{itemize}
We call a pair of oriented edges satisfying at least one of these two conditions a \emph{bad pair} and say that it is \emph{supported} on the corresponding pair of unoriented edges. We use $n_1(G)$ to express  $\gamma_2(\Pc_G)$ explicitly, as follows.
\begin{lemma}\label{lem: gamma1-gamma2 for G}
    Let $G$ be a connected graph. Then, 
    \begin{equation*}
        \gamma_1(\Pc_G)=2\cy(G),
    \end{equation*} 
    and
    \begin{equation}
        \gamma_2(\Pc_G)=2\cy(G)(\cy(G)+2)-n_1(G). \label{eq: gamma2}
    \end{equation}
   \end{lemma}
\begin{proof}
    Let $G=([n],E)$. The next computation shows the first statement:
    \begin{align*}
        \gamma_1(\Pc_G) &= h^\ast_1(\Pc_G)-(n-1)
        = h_1(\Delta_{<})-(n-1)\\
        &= f_0(\Delta_{<})-2(n-1)=2(|E|-n+1)=2\cy(G),
    \end{align*}
    where the first equality follows from the definition of $\gamma_1$ and the fact that $\Pc_G$ is $(n-1)$-dimensional. Using the first statement, we can further show \eqref{eq: gamma2}:
    \begin{align*}
        \gamma_2(\Pc_G) &=h^\ast_2(\Pc_G)-\binom{n-1}{2}-(n-3)\gamma_1(\Pc_G)\\
        &=h_2(\Delta_{<})-\binom{n-1}{2}-2(n-3)\cy(G)\\
        &=f_1(\Delta_{<})-(n-2)f_0(\Delta_{<})+\binom{n-1}{2}-\binom{n-1}{2}-2(n-3)\cy(G)\\
        &=\binom{2|E|}{2}-|E|-n_1(G)-2(n-2)|E|-2(n-3)\cy(G)\\
        &=2|E|(|E|-n+1)-n_1(G)-2(n-3)\cy(G)\\
        &=2\cy(G)(|E|-n+3)-n_1(G)=2\cy(G)(\cy(G)+2)-n_1(G).
    \end{align*}
\end{proof}

Next, we present the main result of this section.

\begin{theorem}\label{thm: gamma2 >0}
Let $G$ be a graph. Then, $\gamma_2(\Pc_G)\geq 0$.
\end{theorem}

The proof of this theorem will require several lemmas and propositions. The strategy is to prove that there exists  an edge $e\in E$ such that $\gamma_2(\Pc_G)\geq \gamma_2(\Pc_{G\setminus e})$, from which the claim follows inductively. We note that, if $e$ is not a bridge of $G$, then $\cy(G\setminus e)=\cy(G)-1$, and \Cref{lem: gamma1-gamma2 for G} directly yields 
    \begin{equation}
        \gamma_2(\Pc_G)-\gamma_2(\Pc_{G\setminus e})=4\cy(G)+2-(n_1(G)-n_1(G\setminus e)). \label{eq: gamma2difference}
    \end{equation}
By the following lemma, we can reduce to the case when $G$ is $2$-connected. 

\begin{lemma}\label{lem: reduce to 2-connected}
    Let $G$ be a graph, and let $G_1,\ldots,G_k$ be its $2$-connected components. Then, $\gamma_1(\Pc_G)=\sum_{i=1}^k \gamma_1(\Pc_{G_i})$, and
    \[
        \gamma_2(\Pc_G) = \sum_{i=1}^k \gamma_2(\Pc_{G_i})+4\sum_{1\leq i < j \leq k} \cy(G_i)\cy(G_j)\geq \sum_{i=1}^k \gamma_2(\Pc_{G_i}).
    \]
\end{lemma}
\begin{proof}
  By \cite[Proposition 4.2]{OT-2021}, the $h^*$-polynomial of $\Pc_G$ is the product of the $h^*$-polynomials of the polytopes $\Pc_{G_i}$. The same holds for their $\gamma$-polynomials, and hence we obtain that $\gamma_2(\Pc_G)=\sum_{i=1}^k (\prod_{j\neq i}\gamma_0(\Pc_{G_j}))\gamma_2(\Pc_{G_i})+\sum_{1\leq i<j\leq k}(\prod_{\ell\neq i,j}\gamma_0(\Pc_{G_{\ell}}))\gamma_1(\Pc_{G_i})\gamma_1(\Pc_{G_j})$. We conclude using that $\gamma_1(\Pc_{G_i})=2\cy(G_i)$ and $\gamma_0(\Pc_{G_i})=1$ hold for every $i$.

\end{proof} 
\
In particular, proving nonnegativity of $\gamma_2(\Pc_G)$ for every $2$-connected graph $G$ is sufficient to prove the statement for every graph.
Our study is divided into cases, which we deal with in Propositions \ref{prop: no 3- 4-cycles}, \ref{prop: deg 2 vertex} and \ref{prop: super long}. We start with the simplest case.
	\begin{proposition}\label{prop: no 3- 4-cycles}
	    Let $G=([n],E)$ be a $2$-connected graph. Assume that there exists $e\in E$ which is not contained in any $3$- or $4$-cycle. Then,
	    \[
	    \gamma_2(\Pc_G)=\gamma_2(\Pc_{G\setminus e})+4\cy(G)+2> \gamma_2(\Pc_{G\setminus e}).
	    \]
	\end{proposition}
	\begin{proof}
    Since $e\in E$ is not contained in any $3$- or $4$-cycle of $G$, its deletion from $G$ does not change the set of $3$- and $4$-cycles of $G$ and, therefore, $n_1(G)=n_1(G\setminus e)$. The claim now follows from \eqref{eq: gamma2difference}.
	\end{proof}

Next, we assume the existence of a vertex of degree $2$. The reason why this case is taken care of separately is that it forces restrictions on which edges can be removed (see \Cref{rem: we need to take deg 2 edge}).
\begin{proposition}\label{prop: deg 2 vertex}
    Let $G=([n],E)$ be a $2$-connected graph. Assume that there exists $e=ij\in E$ such that $\deg_G(i)=2$. Then 
\[
\gamma_2(\Pc_G)\geq \gamma_2(\Pc_{G\setminus e}).
\]
Moreover,  equality holds if and only if every edge of $G$ lies in a $3$- or $4$-cycle together with $e$.
\end{proposition}
\begin{proof}
	If $e$ does not lie in any $3$- or $4$-cycle of $G$, the statement follows from \Cref{prop: no 3- 4-cycles}.
	
	Assume that $e$ is contained in at least one $3$- or $4$-cycle. Let $f=ik$ be the unique edge adjacent to $i$ other than $e$. As $\deg_G(i)=2$, each $3$- or $4$-cycle containing $e$ needs to contain $f$ as well. Hence, if $e$ is contained in some $3$-cycle, then the one on the vertices $i$, $j$ and $k$ is the unique such. Let $s\in\{0,1\}$ and $r\in \N$ be the number of $3$- and $4$-cycles containing $e$ (and hence $f$), respectively. Let $<$ be any order on $E$ for which $e>f>h$ for every $h\in E\setminus\{e,f\}$. By the way $<$ is defined, the minimal element of each $4$-cycle containing $e$ and $f$ is distinct from these.\\
	We now list the bad pairs of $G$ which are not bad pairs of $G\setminus e$. Their number is equal to $n_1(G)-n_1(G\setminus e)$, since every bad pair of $G\setminus e$ is a bad pair of $G$.
	\begin{itemize}
	    \item[-] As $e$ is contained in some $3$- or $4$-cycle, and $e>f>h$ for every $h\in E\setminus\{e,f\}$, the pairs $\{j\to i, i \to k\}$ and $\{k\to i, i \to j\}$ are bad pairs.
	    \item[-] If $s=1$, there are $4$ additional bad pairs, which are contained in an oriented $3$-cycle of $G$, but which are not bad pairs for $G\setminus e$. Namely, the four pairs of oriented edges $\{i\to j, j \to k\}$, $\{k\to j, j\to i\}$, $\{i\to k, k\to j\}$ and $\{j\to k, k\to i\}$. Note that the latter two are not bad  pairs of $G\setminus e$ since they neither lie in a $3$-cycle nor in a $4$-cycle of $G\setminus e$ as $\deg_G(i)=2$. 
	    \item[-] For each $4$-cycle containing $e$, there are $4$ additional pairs. To see this, let $\{ij, j\ell, \ell k, ki\}$ be the edge set of such a $4$-cycle. If the minimal element is $j\ell$, then we get $\{i\to j,\ell \to k\}$, $\{j\to i,k \to \ell\}$, $\{\ell \to k, k \to i\}$ and $\{k\to \ell, i \to k\}$. If instead the minimal element is $\ell k$, then we get $\{i\to j,j \to \ell\}$, $\{j\to i,\ell \to j\}$, $\{j \to \ell, k \to i\}$ and $\{\ell \to j, i \to k\}$. It follows from  $\deg_G(i)=2$, that all of these pairs are not bad pairs of $G\setminus e$.
	\end{itemize}
	We deduce that $n_1(G)-n_1(G\setminus e) = 4r+4s+2$ and hence \eqref{eq: gamma2difference} implies
	\begin{equation}\label{eq: Case 2.1b}
	    \gamma_2(\Pc_G)-\gamma_2(\Pc_{G\setminus e})   =4(\cy(G)-s-r).
	\end{equation}
	To conclude, let $H$ be the subgraph of $G$ consisting of all edges of $G$ which are contained in a $3$- or $4$-cycle together with $e$. By definition $H$ is $2$-connected and $s+r=\cy(H)$. Moreover, since $H$ is a $2$-connected subgraph of the $2$-connected graph $G$, we have that $\cy(G)\geq \cy(H)$. This inequality holds since the cyclomatic number counts the number of ears in any ear decomposition of a graph, and any ear decomposition of $H$ can be completed to one of $G$. Using \eqref{eq: Case 2.1b} this implies $\gamma_2(\Pc_G)-\gamma_2(\Pc_{G\setminus e})\geq 0$.
	
	It remains to characterize the case when $\gamma_2(P_G)=\gamma_2(P_{G\setminus e})$. By the previous argument  $\gamma_2(P_G)=\gamma_2(P_{G\setminus e})=0$ if and only if $\cy(G)=\cy(H)$. As the cyclomatic number of a proper $2$-connected subgraph of $G$ needs to be strictly smaller than $\cy(G)$, it follows that $\cy(G)=\cy(H)$ if and only if $G=H$, which proves the claim. 
	\end{proof}

Before proving the last and main proposition we need a technical lemma.
\begin{lemma}\label{lem: general bound}
    Let $G=([n],E)$ be a $2$-connected graph without vertices of degree $2$  and let $H$ be a $2$-connected subgraph with $k$ vertices of degree $2$. Then, 
    \[
        \cy(G)\geq \cy(H)+\frac{k}{2}.
    \]  
\end{lemma}
\begin{proof}
    Any ear decomposition of $H$ can be completed to one of $G$ by adding $(\cy(G)-\cy(H))$-many ears. Since $G$ does not have vertices of degree $2$, every vertex of degree $2$ in $H$ is adjacent to at least one new ear. Moreover, each new ear is adjacent to at most two such vertices. The number of new ears must then be at least equal to half the number of vertices of degree $2$ in $H$.
\end{proof}
Before handling the remaining case towards the proof of \Cref{thm: gamma2 >0}, we need an additional definition.
\begin{definition}\label{def: cones}
Let $G=([n],E)$ be a graph and let $i$ and $j$ be two new vertices. 
\begin{enumerate}
    \item[(i)] The \emph{double cone} of $G$ with respect to $i$ and $j$ is the graph with vertex set $V\cup\{i,j\}$ and with edges 
    \[
    E\cup(\{i,j\}\times V)\cup\{ij\}.\]
   \item[(ii)] If $G$ is bipartite with bipartition given by $V=V_1\cup V_2$, the the \emph{bipartite cone} of $G$ with respect to $i$ and $j$ is the bipartite graph with vertex set $V\cup\{i,j\}$ and with edges 
    \[
    E\cup(\{i\}\times V_1)\cup (\{j\}\times V_2)\cup\{ij\}.\]
\end{enumerate}
    
\end{definition}

\begin{figure}
    \centering
    \includegraphics[scale=0.7]{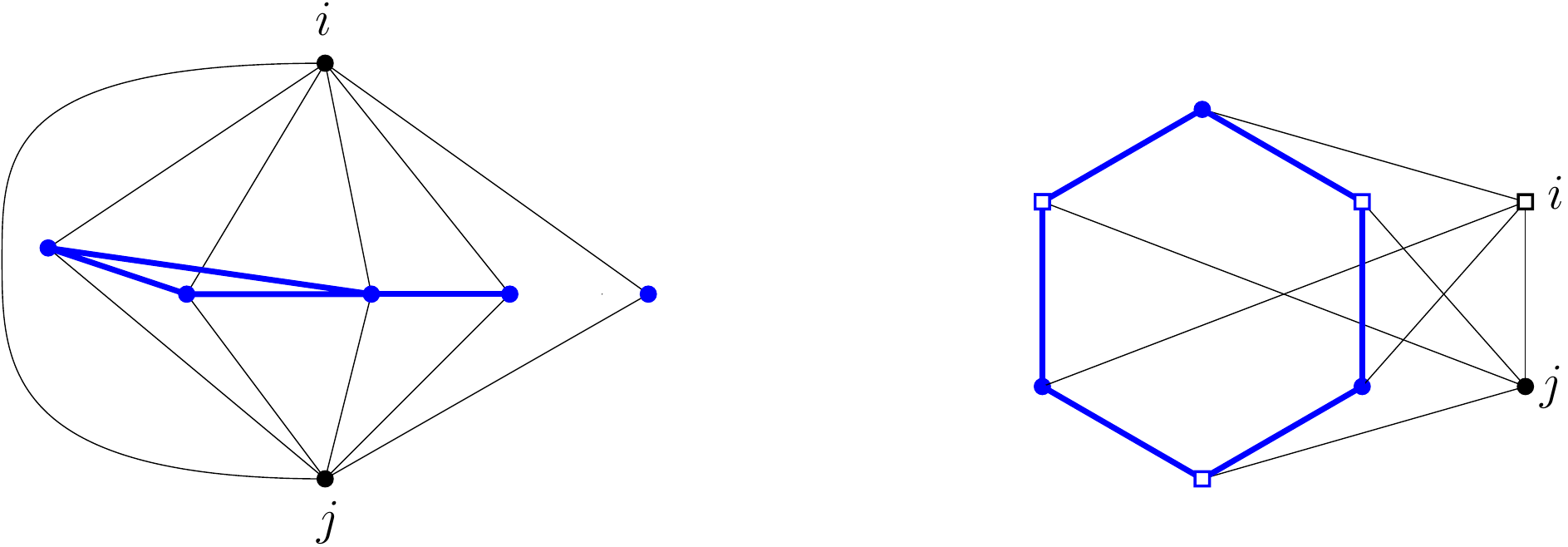}
    \caption{The double cone of a graph (left) and the bipartite cone of a $6$-cycle (right).}
    \label{fig: cones}
\end{figure}

\begin{proposition}\label{prop: super long}
    Let $G=([n],E)$ be a $2$-connected graph. Assume that every edge of $G$ is contained in some $3$- or $4$-cycle, and that $\min_{i\in [n]}\deg_G(i)\geq 3$. Then, for every $e\in E$,
    \[
    \gamma_2(\Pc_{G})\geq \gamma_2(\Pc_{G\setminus e}).
    \]
    Moreover, equality holds if and only if $G=G_1\oplus_2\cdots \oplus_2 G_m$, where $m \geq 1$, all the $2$-clique sums are taken along $e$, $G_1$ is the double cone w.r.t. $i$ and $j$ over a connected graph $G_1'$ with $|E(G_1')| \geq 1$, and for every $2\leq \ell\leq m$ either:
    \begin{itemize}
        \item[-] $G_\ell$ is the double cone w.r.t. $i$ and $j$ over any connected graph $G_\ell'$ with $|E(G_{\ell}')| \geq 1$, or
        \item[-] $G_\ell$ is the bipartite cone w.r.t. $i$ and $j$ over an even cycle.
    \end{itemize}
\end{proposition}
\begin{proof}
	Let $e=ij$ be any edge. We define $H=(W,F)$ to be the subgraph induced by all edges of $G$ which lie in a $3$- or $4$-cycle together with $e$. We now give an iterative procedure to construct $H$ in a sequence of steps, yielding a partition of $F$. We use $H'=(W',F')$ to denote the current graph in the procedure. Set $E'=E$.
	\begin{itemize}
		\item[\textbf{Step 0}:] Set $H'=(\{i,j\},\{ij\})$, namely $H'$ is the graph consisting of the edge $e$ alone.
		\item[\textbf{Step 1}:] For every pair of edges $f,g\in E'$ such that $\{e,f,g\}$ is a $3$-cycle, add $f,g$ to $F'$ and delete them from $E'$. This step adds to $H'$ a number $r_1$ of ears of length $2$.
		\item[\textbf{Step 2}:] Add to $F'$ every edge $k\ell\in E'$ such that $k$ and $\ell$ are vertices of $H'$, and delete these edges from $E'$. This step adds to $H'$ a number $r_2$ of ears of length $1$;
		\item[\textbf{Step 3}:] If there is a $4$-cycle $C$ in $G$ with $E(C)\cap F'=\{e\}$, add the three edges in $E(C)\setminus\{e\}$ to $F'$, and delete them from $E'$. Update $H'$ and repeat this procedure as  often as possible. In this step, $r_3$ many ears of length $3$ are added to $H'$.
		\item[\textbf{Step 4}:] If there is a $4$-cycle $C$ in $G$ with $E(C)\cap F'=\{e,g\}$ for some edge $g$, add $E(C)\setminus\{e,g\}$ to $F'$ and delete these edges from $E'$. Update $H'$ and iterate this  procedure as long as possible. This step adds to $H$' a number $r_4$ of ears of length $2$;	
		\item[\textbf{Step 5}:] Add to $F'$ the edges $f\in E'$ such that $e$ and $f$ are contained in a $4$-cycle $C$ with $E(C)\setminus F'=\{f\}$. This step adds to $H'$ a number $r_5$ of ears of length $1$.
	\end{itemize} 
	\begin{figure}[h]
		\centering
		\includegraphics[scale = 0.75]{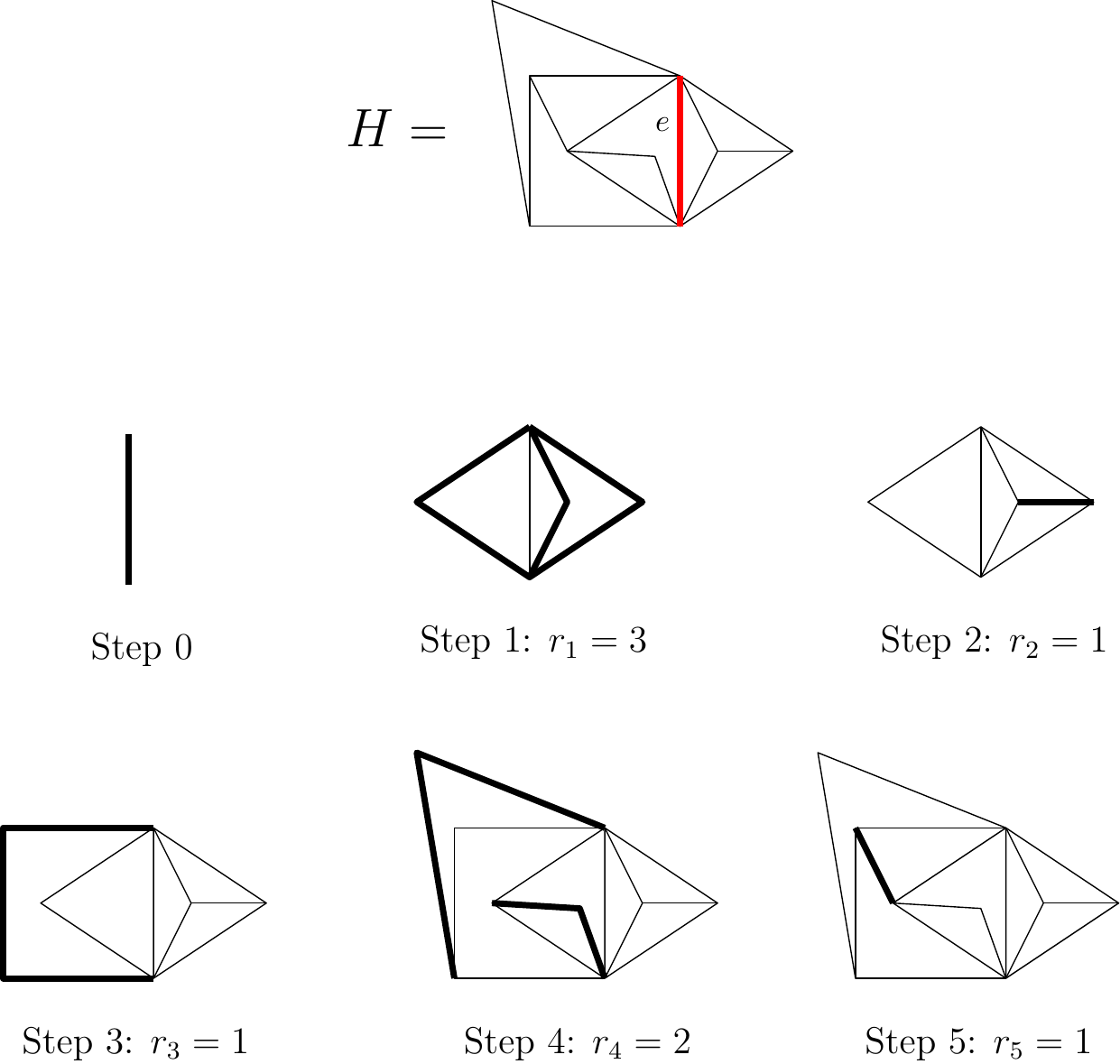}
		\caption{The construction of the graph $H$ as in the proof of \Cref{prop: super long}.}
		\label{fig: algorithm for H}
	\end{figure}
	(See \Cref{fig: algorithm for H} for an example of how this algorithm works.) It is obvious that this procedure indeed yields an open ear decomposition of $H$ (the closed ear being the first cycle that is constructed either in Step 1 or 3). Hence,
	\begin{equation}\label{eq: cyclomaticNumberH}
	\cy(H)=\sum_{i=1}^5 r_i.
	\end{equation}
	Observe that we make multiple choices in Steps $3$ and $4$ and hence neither the decomposition nor the numbers $r_3$, $r_4$ and $r_5$ are uniquely determined. 
	Fix now any linear order on $F$ such that:
	\begin{itemize}
	    \item[-] $g$ is bigger than $f$ if $g$ has been added to $F$ before $f$;
	    \item[-] if $f,g,h$ are edges added in an iteration of Step $3$, then the smallest of the three is the one not incident to $e$;
	    \item[-] if $f,g$ are edges added in an iteration of Step $4$, then the smallest of the two is the one not incident to $e$.
	\end{itemize}
	 Consider an extension $<$ of this linear order to $E$ such that any edge of $E\setminus F$ is smaller than any edge of $F$. In particular, $e$ is the $<$-maximal edge.
	 
	We now describe the bad pairs of $G$ which are not bad pairs of $G\setminus e$. We will use the fact that any pair of disjoint oriented edges determines a unique oriented $4$-cycle. In the following, let $\mathbbm{1}_{C}$ be the indicator function which equals 1 if condition $C$ holds and 0 otherwise.  
	\begin{itemize}
	    \item[-] For every edge $f$ lying in a $3$-cycle with $e$, there are  $2$ bad pairs supported on $\{e,f\}$. As there are $2r_1$ such edges $f$, this gives rise to $4r_1$ bad pairs. 
	    \item[-] There are $2$ bad pairs supported on the set $\{f,g\}$, where $\{e,f,g\}$ is the last $3$-cycle added in Step 1. As these pairs only occur if $r_1\geq 1$, their number is $2\cdot\mathbbm{1}_{r_1\geq 1}$.
	    \item[-] Each edge $h$ added in Step 2 lies in a subgraph of $H$ isomorphic to $K_4$ that also contains $e$. In particular, there are two $4$-cycles containing $e$ and $h$, and in both cycles $h$ is the minimal element. For each cycle, only the two edges different from $e$ and $h$ give rise to $2$ new bad pairs. Therefore, for each edge $h$ added in Step 2 we get $4$ new bad pairs, and hence we obtain $4r_2$ many.
	    \item[-] If $f,g$ and $h$ have been added in the same iteration of Step 3 with $\min_{<}\{e,f,g,h\}=h$, then there are $2$ bad pairs supported on each of $\{e,f\}$, $\{e,g\}$ and $\{f,g\}$. This yields $6r_3$ such bad pairs.
	    \item[-] If $\{e,f,g,h\}$ lie in a $4$-cycle such that $g$ and $h$ have been added in the same iteration of Step 4 and $h<g$, then there are $2$ bad pairs supported on each of $\{e,g\}$ and $\{f,g\}$. Note that $f$ is incident to $e$ and has been added either in Step $1$ or in Step $3$. This implies that the two bad pairs supported on $\{e,f\}$ have already been counted in the previous discussion. Hence there are $4r_4$ new bad pairs.
	    \item[-] If  $\{e,f,g,h\}$ is a $4$-cycle such that $h$ has been added in Step $5$, then there are $2$ bad pairs supported on $\{f,g\}$. These are new, since  $f$ and $g$ did not lie in a $4$-cycle with $e$ before Step $5$. There are $2r_5$ such bad pairs.
	\end{itemize}
		 
	\begin{figure}[h]
	\centering
	\includegraphics[scale = 0.8]{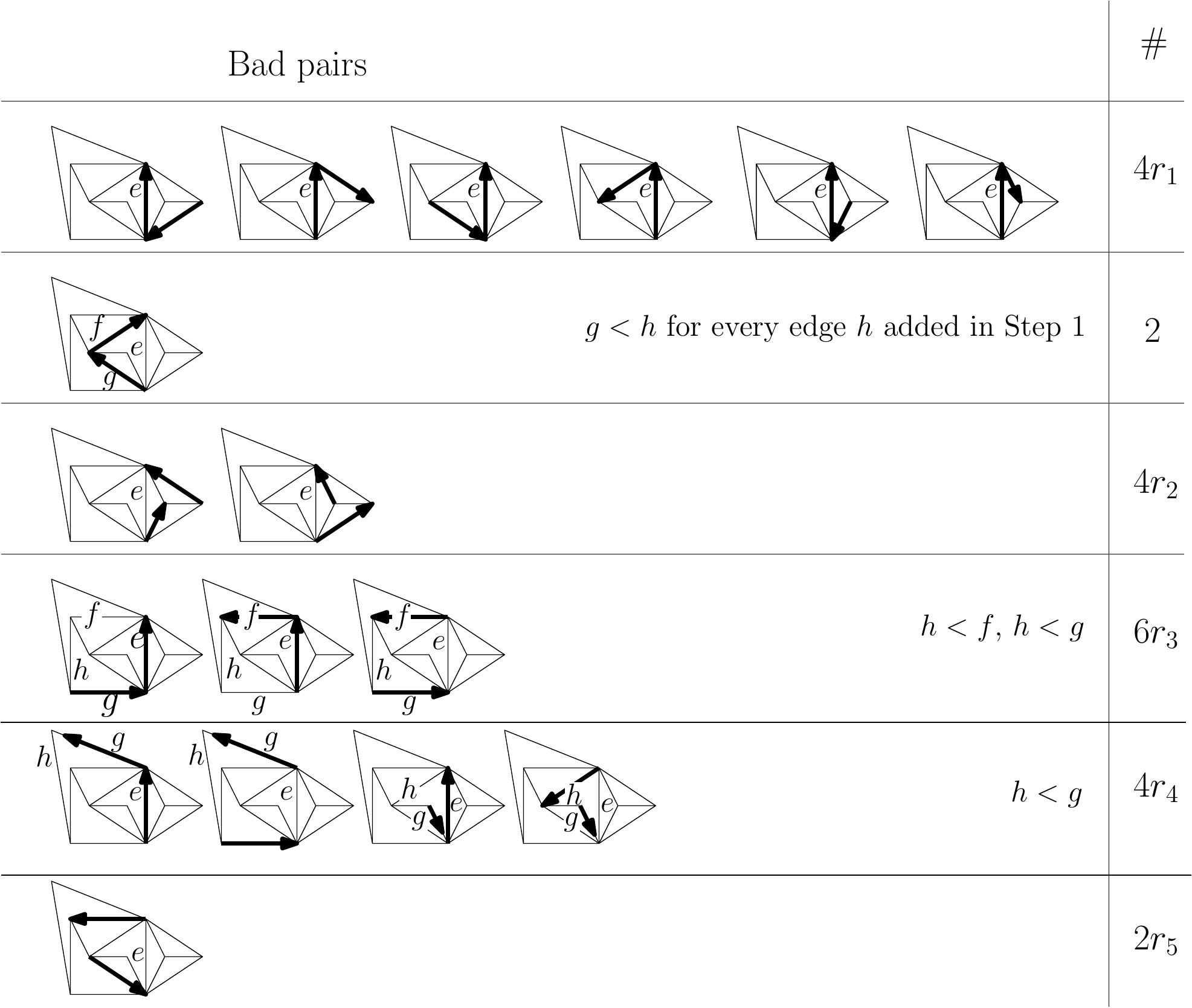} 
	\caption{The bad pairs of edges as in \Cref{prop: super long} for the graph $H$ in \Cref{fig: algorithm for H}. }
	\label{fig: lost and found}
\end{figure}
	\Cref{fig: lost and found} shows all bad pairs for the graph $H$ in \Cref{fig: algorithm for H}.
	For the total number of bad pairs in $G$ that are not bad pairs of $G\setminus e$ we hence get  
	\begin{align*}
	n_1(G)-n_1(G\setminus e)&=4r_1+2\cdot\mathbbm{1}_{r_1\geq 1}+4r_2 +6r_3+4r_4+2r_5\\
	    &=4\cy(H) +2(r_3-r_5) + 2\cdot\mathbbm{1}_{r_1\geq 1},
	\end{align*}
	where the second equality follows from \eqref{eq: cyclomaticNumberH}.
	Hence, \eqref{eq: gamma2difference} implies
	\begin{equation}\label{eq: Case 2.2b}
	    \gamma_2(\Pc_G)-\gamma_2(\Pc_{G\setminus e})=4(\cy(G)-\cy(H)) - 2(r_3-r_5)+2(1-\mathbbm{1}_{r_1\geq 1}).
	\end{equation}

As $H$ is $2$-connected, we have that $\cy(G)-\cy(H)\geq 0$, and hence the only possibly negative term in the last equation is $-2(r_3-r_5)$. In particular, if $r_3-r_5\leq 0$, it follows that $\gamma_2(\Pc_G)-\gamma_2(\Pc_{G\setminus e})\geq 0$. Now assume $r_3-r_5\geq 0$. We claim that in this case
	\begin{align}\label{eq: stronger ineq}
	    \cy(G)\geq \cy(H) + (r_3-r_5).
	\end{align}
	Note that using \eqref{eq: Case 2.2b} it then follows that 
	\begin{equation}\label{eq:gammaDifference2}
	    \gamma_2(\Pc_G)-\gamma_2(\Pc_{G\setminus e})\geq 2(r_3-r_5)+2(1-\mathbbm{1}_{r_1\geq 1}),
	\end{equation}
which is even stronger than $\gamma_2(\Pc_G)-\gamma_2(\Pc_{G\setminus e})\geq 0$.

	To show \eqref{eq: stronger ineq}, let $J=H\setminus\{i,j\}$ (i.e.,~$J$ is the graph obtained by removing the vertices $i$ and $j$ from $H$). The vertex set of $J$ can be partitioned as $V(J)=V_1\cup V_3\cup V_4$, where $V_{\ell}$ is the set of vertices that have been added to $H$ during Step $\ell$ of the described procedure. Since Steps $2$ and $5$ add ears of length $1$, no new vertices are introduced during these steps. Note that $|E(J)|=r_2+r_3+r_4+r_5$. For each connected component $J_{\ell}$ of $J$, we let $r_{k,\ell}$ be the number of edges of $J_{\ell}$ added to $H$ in Step $k$, for $k=1,\ldots,5$. We distinguish between two cases:
	\begin{itemize}
	    \item[\sf{Case 1:}] If $V(J_{\ell})\cap V_1\neq \emptyset$, then we consider an auxiliary graph $J_{\ell}'$ with vertex set $V(J_{\ell}')=(V(J_{\ell})\setminus V_1)\cup \{v\}$, where $v$ is a new vertex. Two vertices $a,b\in V(J_{\ell}')$ form an edge of $J_{\ell}'$ if either $ab\in E(J_{\ell})$ or $a=v$ and $wb\in E(J_{\ell})$, for some $w\in V(J_{\ell})\cap V_1$. Then $J_{\ell}'$ is connected. Hence,
	    \begin{align}\label{eq: Vi cap V1 nonempty}   r_{3,{\ell}}+r_{4,{\ell}}+r_{5,{\ell}}\geq|E(J_{\ell}')|\geq |V(J_{\ell}')|-1 = 2r_{3,{\ell}}+r_{4,{\ell}}+1-1,
	    \end{align}
	    where the first inequality follows from the fact that the edges in $J_{\ell}$ between two vertices in $V_1$ do not appear in $J_{\ell}'$. We obtain that $r_{3,{\ell}}-r_{5,{\ell}}\leq 0$.
	    			 
	    \item[\sf{Case 2:}] If $V(J_{\ell})\cap V_1=\emptyset$, then
	    \begin{align}\label{eq: Vi cap V1 empty}	        r_{3,{\ell}}+r_{4,{\ell}}+r_{5,{\ell}}=|E(J_{\ell})|\geq |V(J_{\ell})|-1 = 2r_{3,{\ell}}+r_{4,{\ell}}-1.
	    \end{align}
	    This implies $r_{3,{\ell}}-r_{5,{\ell}}\leq 1$, with equality attained if and only if $|E(J_{\ell})|=|V(J_{\ell})|-1$, i.e., if $J_{\ell}$ is a tree (with at least one edge, since  $r_{3,{\ell}}=r_{5,{\ell}}+1\geq 1$). In this case, $J_{\ell}$ has at least $2$ leaves. Since each leaf of $J_{\ell}$ corresponds to a vertex of degree $2$ in $H$, one has that \[|\{v\in H~:~ \deg_H(v)=2\}|\geq 2 \cdot |\{{\ell} ~ : ~J_{\ell} \text{ is a tree with }|V(J_{\ell})|\geq 2 \text{ and } V(J_{\ell})\cap V_1=\emptyset\}|.
	    \]
	\end{itemize}

	Combining the two cases above and using the identities $\sum_{\ell} r_{k,{\ell}}=r_k$ we obtain that
	\begin{equation} \label{eq:trees}
	    r_3-r_5\leq |\{{\ell} ~ : ~J_{\ell} \text{ is a tree with } |V(J_{\ell})|\geq 2, \,V(J_{\ell})\cap V_1=\emptyset\}|\leq 
	    \dfrac{|\{v\in H~:~ \deg_H(v)=2\}|}{2}.
	\end{equation}
	Using \Cref{lem: general bound} we conclude that
	\[
	    \cy(G)\geq \cy(H)+\dfrac{|\{v\in H~:~ \deg_H(v)=2\}|}{2}\geq \cy(H) + (r_3-r_5),
	\]
	which proves \eqref{eq: stronger ineq} and hence the inequality $\gamma_2(P_G)-\gamma_2(P_{G\setminus e})\geq 0$.
	
	We now study the case when $\gamma_2(P_G)-\gamma_2(P_{G\setminus e})= 0$. It follows from \eqref{eq: Case 2.2b}, \eqref{eq: stronger ineq} and \eqref{eq:gammaDifference2} that $\gamma_2(P_G)-\gamma_2(P_{G\setminus e})= 0$ if and only if $r_3=r_5$, $r_1\geq 1$ and $\cy(G)=\cy(H)$. The last equality implies that $G=H$. In particular, since we assumed that every vertex in $G$ has degree at least $3$, the same holds for $H$. It follows from \eqref{eq:trees} that there is no component $J_{\ell}$ with $V(J_{\ell})\cap V_1= \emptyset$ that is a tree with at least one edge. In particular, we have that $r_{3,{\ell}}-r_{5,{\ell}}\leq 0$ for any component $J_{\ell}$ and, as $r_3=r_5$, all of these inequalities are in fact equalities. The idea now is to analyze what the components $J_{\ell}$ can look like.

If $V(J_{\ell})\cap V_1\neq \emptyset$, then it follows from \eqref{eq: Vi cap V1 nonempty} that $J_{\ell}'$ has to be a tree and that we have $|E(J_{\ell}')|=r_{3,{\ell}}+r_{4,{\ell}}+r_{5,{\ell}}$. It follows from these two conditions that for every $w\in V(J_{\ell})\cap (V_3\cup V_4)$ there is at most one edge of the form $wz$ for some $z\in V(J_{\ell})\cap V_1$. This implies that every vertex of degree $1$ in $J_{\ell}'$ (other than $v$) corresponds to a vertex of degree $2$ in $H$. As there are none such, we conclude that $J_{\ell}'$ is just an isolated vertex, namely $v$. Hence, $r_{3,{\ell}}=r_{4,{\ell}}=r_{5,{\ell}}=0$ and $J_{\ell}$ is an arbitrary graph on $r_{1,{\ell}}$ vertices, where each vertex is connected to both $i$ and $j$ in $G$. This implies that the subgraph of $G$ induced on $V(J_{\ell})\cup\{i,j\}$ is the double cone over $J_{\ell}$ w.r.t. $i$ and $j$.

	If $V(J_{\ell})\cap V_1=\emptyset$, then by \eqref{eq: Vi cap V1 empty} we have that $r_{3,{\ell}}-r_{5,{\ell}}=0$ if and only if $|E(J_{\ell})|=|V(J_{\ell})|$. By the same argument as above, $J_{\ell}$ cannot have vertices of degree $1$ and must hence be a cycle. Since all vertices of $J_{\ell}$ have been added in an iteration of Step $3$ or $4$, each vertex of $J_{\ell}$ is connected to either $i$ or $j$ in $G$, but not to both. Consider an edge $vw\in E(J_{\ell})$ and assume that $vi\in E(G)$. As $vw$ lies in a $4$-cycle together with $e$ by assumption and there is a unique such, namely the one with edges $\{vw,vi,e, wj\}$, it follows that $wj\in E$. This shows that $J_{\ell}$ is a bipartite graph with vertex partition $\{v\in V(J_{\ell})~:~vi\in E\}\cup\{v\in V(J_{\ell})~:~vj\in E\}$. Hence, $J_{\ell}$ is an even cycle, and the subgraph of $G$ induced on $V(J_{\ell})\cup\{i,j\}$ is the bipartite cone of $J_{\ell}$ w.r.t. $i$ and $j$.

\end{proof}

\noindent We can finally provide the proof of \Cref{thm: gamma2 >0}.
\begin{proof}[Proof of \Cref{thm: gamma2 >0}]
Let $G=([n],E)$ be a graph. We show the claim by induction on $|E|$. If $|E|=1$, the claim is trivially true. Assume that $|E|>1$. Without loss of generality, we can assume that $G$ is connected since taking $1$-sums of its connected components does not change the symmetric edge polytope \cite[Remark 4.8]{DDM} and, in particular, its $\gamma$-vector. If $G$ is not $2$-connected, then let $G_1,\ldots,G_s$ be its $2$-connected components, where $s\geq 2$. As $|E(G_i)|<|E(G)|$, it follows from the induction hypothesis and \Cref{lem: reduce to 2-connected} that $\gamma_2(\Pc_G)\geq 0$. Assume that $G$ is $2$-connected. Applying Propositions \ref{prop: no 3- 4-cycles}, \ref{prop: deg 2 vertex} and \ref{prop: super long}, it follows that there exists $e\in E$ with $\gamma_2(\Pc_G)\geq \gamma_2(\Pc_{G\setminus e})$. Since, by the induction hypothesis, the latter expression is nonnegative, this finishes the proof.
\end{proof}

\begin{remark}\label{rem: we need to take deg 2 edge}
	We remark that in the proof of \Cref{thm: gamma2 >0} we do not claim that $\gamma_2(\Pc_G)\geq \gamma_2(\Pc_{G\setminus e})$ for \emph{every} edge $e\in E$. This statement is indeed false. For a small counterexample, let $G$ be the $2$-connected graph on $5$ vertices with $E=\{12,23,34,45,15,35\}$. Then $\gamma_2(\Pc_G)=4$ and $\gamma_2(\Pc_{G\setminus 35})=6$, so $\gamma_2$ increases when removing the edge $35$. However, $\deg_G(3)=\deg_G(5)=3$ and $\deg_G(1)=\deg_G(2)=\deg_G(4)=2$ in $G$, so this is the setting from \Cref{prop: deg 2 vertex}. In particular, the proof states that we should choose $e$ to be adjacent to one of the vertices of degree $2$, a condition that the edge $35$ does not satisfy.
\end{remark}

In the remaining part of this section, we focus on the problem of when $\gamma_2(\Pc_G)=0$. 
\label{def: Gn}
\begin{definition}
	Let $n\geq 3$. Let $G_n$  be the graph on $n$ vertices obtained from the complete bipartite graph $K_{2,n-2}$, considered with bipartition $[n]=[2]\cup\{3,\ldots,n\}$, by adding the edge $12$.
\end{definition}

We note that $G_n$ has $2n-3$ edges.


\begin{theorem}\label{cor: gamma_2=0}
	Let $G=([n],E)$ be a $2$-connected graph. Then $\gamma_2(\Pc_G)=0$ if and only if either $n<5$, or $n\geq 5$ and $G\cong G_n$ or $G\cong K_{2,n-2}$.
\end{theorem}

\begin{proof}
By \cite[Example 5.9]{OT-2021} we have that $\gamma_i(\Pc_{K_n})= \binom{n-1}{2i} \binom{2i}{i}$, and in \cite{HJM} it is proved that $\gamma_i(\Pc_{K_{m,n}}) = \binom{m-1}{2i} \binom{n-1}{2i} \binom{2i}{i}$. Moreover, $\gamma_i(\Pc_{G_n})= \gamma_i(\Pc_{K_{2,n-1}})$ for every $i\geq 0$ by \cite[Proposition 5.4]{OT-2021}. This proves the ``if" statement, which can be also verified directly using \Cref{lem: gamma1-gamma2 for G}, \eqref{eq: E K} and \eqref{eq: E G}. 

    We prove the claim by double induction on the pairs $(n,k)$, with $n=|V(G)|$ and $k=|E(G)|$. 
    The base case is given by any graph with $n<5$, as $\gamma_2(\Pc_G)=0$ for every graph $G$ with less than $5$ vertices.
    
    Let $n\geq 5$. First assume that $G$ has a vertex $i$ of degree $2$. Let $e=ij$ be any of the two edges incident with $i$. Being $2$-connected, $G$ is either a cycle (in which case \Cref{thm: gamma2 >0} and \Cref{prop: deg 2 vertex} imply that $n\in \{3,4\}$) or can be obtained from a $2$-connected graph $G'$ by adding an open ear $P$ of length $\ell \geq 2$ containing $e$. Assume the latter. We then have that
    

    \begin{equation*}
    0=\gamma_2(\Pc_G)\geq \gamma_2(\Pc_{G\setminus e})=\gamma_2(\Pc_{G'})\geq 0,
    \end{equation*}
    where the first inequality follows from \Cref{prop: deg 2 vertex}, the last equality from \eqref{eq: gamma2} in \Cref{lem: gamma1-gamma2 for G}, and the last inequality from \Cref{thm: gamma2 >0}.
    Hence $\gamma_2(\Pc_{G'})=0$ and, by induction, $G' \in \{G_{n'} : n' \geq 3\} \cup \{K_{2,n'-2} : n' \geq 4\} \cup \{K_4\}$. We claim that the length $\ell$ of the ear $P$ must be $2$. Indeed, consider any edge $f$ in $G'$ different from the one (if it exists) connecting the two endpoints of $P$. Then every cycle containing both $e$ and $f$ has length at least $\ell + 2$, and \Cref{prop: deg 2 vertex} forces $\ell = 2$. Let then $e=ij$, $ik$ be the edges in $P$. Note that $G'$ cannot be a $K_4$, as otherwise the edge of $K_4$ opposite to $jk$ would not be contained in any 3- or 4-cycle together with $e$. Then either $G' \cong G_{n-1}$ or $G' \cong K_{2,n-3}$, and thus $G\cong G_n$ (if $jk\in E$) or $G\cong K_{2,n-2}$ (otherwise).
    
    If $\min\text{deg}_G(v)\geq 3$, we choose $e$ to be the unique edge in the last ear of any ear decomposition of $G$. We then have that $G\setminus e$ is $2$-connected and has $n \geq 5$ vertices. As $\gamma_2(\Pc_{G\setminus e})=0$, we conclude by induction that $G\setminus e\cong K_{2,n-2}$ or $G\setminus e\cong G_n$. In both cases, $G\setminus e$ has at least $3$ vertices of degree $2$. Hence $G$ has at least one degree $2$ vertex, which contradicts the assumption $\min\text{deg}_G(v)\geq 3$.
\end{proof}
The characterization of the equality case $\gamma_2=0$ can be extended to all graphs as follows.
\begin{corollary}\label{cor:gamma2=0general}
    Let $G=([n],E)$ be a graph. Then $\gamma_2(\Pc_G)=0$ if and only if either
    \begin{enumerate}
        \item[(i)] $G$ is a forest, or,
        \item[(ii)] all but one of the $2$-connected components of $G$ are edges and the remaining component is isomorphic to one of $K_4$, $G_\ell$ for some $\ell\geq 3$, and $K_{2,\ell}$ for some $\ell\geq 2$.
    \end{enumerate} 
\end{corollary}
\begin{proof}
    Assume $\gamma_2(\Pc_G)=0$. If $G$ has at least two $2$-connected components that are not edges, then the product of their cyclomatic numbers is positive, as the cyclomatic number of any $2$-connected graph is strictly positive. By \Cref{lem: reduce to 2-connected} and \Cref{thm: gamma2 >0}, this implies that $\gamma_2(\Pc_G)>0$. Hence, $G$ has at most one $2$-connected component that is not an edge. Let us denote this component by $H$, if it exists. Again using \Cref{lem: reduce to 2-connected} we observe that if $\gamma_2(\Pc_H)>0$, then $\gamma_2(\Pc_G)>0$. The claim now follows from \Cref{cor: gamma_2=0} by noting that the only $2$-connected graphs on less than $5$ vertices are $K_4$, $C_3=G_3$, $C_4=K_{2,2}$ and $G_4$.
\end{proof}
\begin{figure}[h]
	\centering
	\includegraphics[scale=0.75]{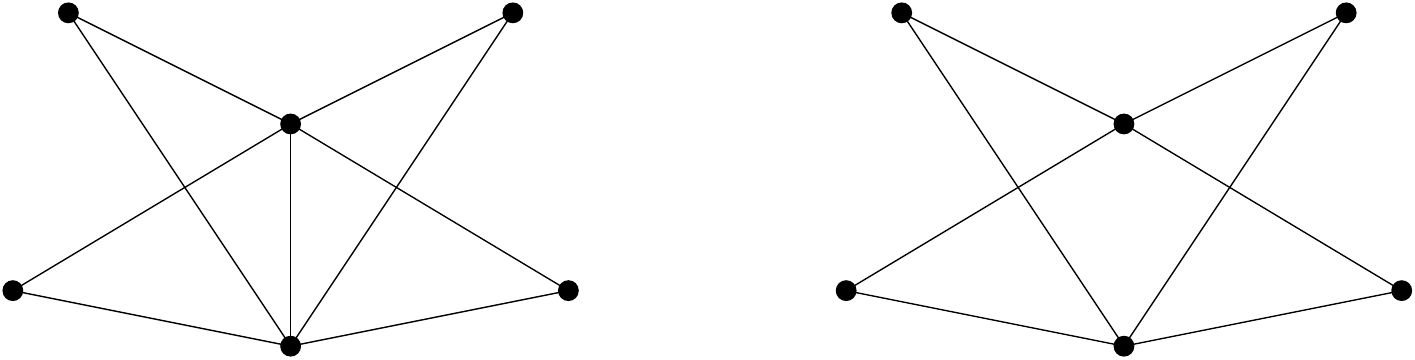}
	\caption{The graphs $G_6$ and $K_{2,4}$.}
	\label{fig:my_label}
\end{figure}

We close this section with a conjecture that extends \Cref{cor: gamma_2=0}. To state the conjecture, for $k\geq 2$, let $G_{n,k}$ be the graph that is  obtained from $K_{k,n-k}$ by adding all edges between the vertices on the side of the vertex partition with $k$ elements. In other words, $G_{n,k}$ can be thought of as the \emph{$k$-fold cone} over a set of $n-k$ isolated vertices. In particular, for $k=2$ we have $G_{n,2}=G_n$. The following conjecture naturally generalizes \Cref{cor: gamma_2=0} and has been verified computationally for small values of $k$ and $n$.

\begin{conjecture}
Let $k\in \mathbb{N}$ and let $G$ be a $k$-connected graph on $n$ vertices. Then $\gamma_k(\Pc_G)=0$ if and only if $n<2k+1$ or, $n\geq 2k+1$ and  $K_{k,n-k}\subseteq G\subseteq G_{n,k}$.
\end{conjecture}

\section{On a conjecture of Lutz and Nevo} \label{sec:lutznevo}
In this section, our focus lies on a conjecture by Lutz and Nevo \cite[Conjecture 6.1]{LN} which characterizes flag PL-spheres $\Delta$ with $\gamma_2(\Delta)=0$. Our goal is to show that symmetric edge polytopes with $\gamma_2=0$ admit a triangulation of their boundary with the properties predicted by \cite[Conjecture 6.1]{LN}.

\begin{conjecture}\cite[Conjecture 6.1]{LN}\label{conj:LN}
    Let $\Delta$ be a $(d-1)$-dimensional flag PL-sphere, with $d\geq 4$. Then the following are equivalent:
    \begin{itemize}
        \item[(i)] $\gamma_2(\Delta)=0$;
        \item[(ii)] There exists a sequence of edge contractions
        \[
            \Delta=\Delta_0\to \Delta_1=\Delta_0/F_1\to \cdots \to \Delta_{k-1}/F_k\cong \Diamond_{d-1},
        \]
        such that each $\Delta_i$ is a $(d-1)$-dimensional flag PL-sphere, and $\lk_{\Delta_{i-1}}(F_{i})\cong \Diamond_{d-3}$, for every $1\leq i \leq k$.
    \end{itemize}
\end{conjecture} 

    The implication ``(ii)$\Rightarrow $(i)'' follows easily from the well-known fact that the $\gamma_1=\gamma_2=0$ for the boundary of any cross-polytope combined with the following relation between the $\gamma$-vectors of $\Delta$ and of an edge contraction $\Delta/F$:
\begin{equation}\label{eq:gammaContraction}
    \gamma_2(\Delta) = \gamma_{2}(\Delta/F) + \gamma_{1}(\lk_{\Delta}(F)).
\end{equation}
The remaining implication has been proven for the subclass of (dual complexes) of flag nestohedra (see \cite[Section 6]{LN} and \cite{Volodin}) and has been tested computationally by Lutz and Nevo \cite[Section 6]{LN}, but is open in general. 

In the following, we show that the boundary complexes of the symmetric edge polytopes of the graphs $K_{2,n-2}$ and $G_n$ admit a triangulation satisfying (i) and (ii) above. We start by fixing labelings on $K_{2,n-2}$ and $G_n$. We label the vertices of $K_{2,n-2}$ and $G_n$ so that $E(K_{2,n-2})=\{1,2\}\times \{3,\dots,n\}$ and $E(G_n)=E(K_{2,n-2})\cup \{12\}$. Let further $<$ be a total order on the edges of both graphs such that $2n<2(n-1)<\dots<23$ are the smallest edges and let  $\Delta_{K_{2,n-2}}$ and $\Delta_{G_n}$ be the corresponding unimodular triangulations of $\partial\Pc_{K_{2,n-2}}$ and $\partial\Pc_{G_n}$, respectively, provided by \Cref{lem: nonfaces of Delta<}. The $6n^2-28n+34$ edges of $\Delta_{K_{2,n-2}}$ and the $6n^2-24n+24$ edges of $\Delta_{G_n}$ can then be listed as follows:

\begin{align}\label{eq: E K}
    E(\Delta_{K_{2,n-2}})=&\{\pm\{e_{i,a},e_{i,b}\}~:~ 1\leq i\leq 2, 3\leq a < b\leq n\} \cup\nonumber\\
    &\{ \pm\{e_{1,a},e_{b,2}\}~:~ 3\leq a \neq b\leq n\} \cup \nonumber\\
    &\{\pm\{e_{1,a},e_{2,a}\}~:~ 3\leq a\leq n\} \cup\\
    &\{\pm\{e_{2,b},e_{1,a}\}~:~ 3\leq a <b\leq  n\}\cup\nonumber\\
    &\{\pm\{e_{2,a},e_{b,2}\} ~:~ 3\leq a< b\leq n\}\cup\nonumber \\
    &\{\pm\{e_{1,n},e_{n,2}\}\}.\nonumber
\end{align}

\begin{align}\label{eq: E G}
    E(\Delta_{G_n})=&\{\pm\{e_{i,a},e_{i,b}\}~:~ 1\leq i\leq 2, 3\leq a < b\leq n\} \cup\nonumber\\
    &\{ \pm\{e_{1,a},e_{b,2}\}~:~ 3\leq a \neq b\leq n\} \cup \nonumber\\
    &\{\pm\{e_{1,a},e_{2,a}\}~:~ 3\leq a\leq n\} \cup\\
    &\{\pm\{e_{2,b},e_{1,a}\}~:~ 3\leq a <b\leq  n\}\cup\nonumber\\
    &\{\pm\{e_{2,a},e_{b,2}\} ~:~ 3\leq a< b\leq n\}\cup\nonumber \\
    &\{\pm\{e_{1,2},e_{1,a}\},\pm\{e_{1,2},e_{a,2}\}~:~ 3\leq a \leq n\}.\nonumber
\end{align}
In particular, we get that $E(\Delta_{K_{2,n-2}})\setminus E(\Delta_{G_n})=\{\pm\{e_{1,n},e_{n,2}\}\}$, and the only edges of $\Delta_{G_n}$ that are non-edges of $\Delta_{K_{2,n-2}}$ are those containing $e_{1,2}$ or $e_{2,1}$. 

\begin{lemma}\label{lem: two steps}
For every $n\geq 3$ we have:
\begin{itemize}
    \item[(i)] $\Delta_{K_{2,n-2}}\cong\langle e_{2,n},e_{n,2}\rangle*\Delta_{G_{n-1}}$,
    \item[(ii)] $(\Delta_{G_n}/\{e_{1,2},e_{1,n}\})/\{e_{2,1},e_{n,1}\}\cong\Delta_{K_{2,n-2}}$.
\end{itemize}
\end{lemma}
Observe that $\{e_{1,2},e_{1,n}\}$ is an edge of $\Delta_{G_n}$ and $\{e_{2,1},e_{n,1}\}$ is an edge of $\Delta_{G_n}/\{e_{1,2},e_{1,n}\}$. Hence it makes sense to consider the corresponding edge contractions in (ii).

\begin{proof} 
To prove (i) we first note that, since by \Cref{lem: nonfaces of Delta<} both complexes involved in the statement are flag spheres, it suffices to provide an isomorphism between the $1$-skeleta of the corresponding complexes. For this aim, let $\varphi: \Delta_{K_{2,n-2}}\to \langle e_{2,n},e_{n,2}\rangle*\Delta_{G_{n-1}}$ be the simplicial map induced by $\varphi(\pm e_{1,n})=\pm e_{1,2}$ and  $\varphi(e)=e$ for any other vertex $e\in \Delta_{K_{2,n-2}}$. By comparing \eqref{eq: E K} and \eqref{eq: E G}, it is easily seen that $\varphi$ is a simplicial isomorphism between the $1$-skeleta of $\Delta_{K_{2,n-2}}$ and $\langle e_{2,n},e_{n,2}\rangle*\Delta_{G_{n-1}}$.

To show (ii), observe that by \Cref{lem: nonfaces of Delta<} $\Delta_{K_{2,n-2}}$ and $\Delta_{G_n}$ are flag simplicial complexes. We first show that so is $(\Delta_{G_n}/\{e_{1,2},e_{1,n}\})/\{e_{2,1},e_{n,1}\}$. For this it is enough to show that $\{e_{1,2},e_{1,n}\}$ and $\{e_{2,1},e_{n,1}\}$ are not contained in any induced subcomplex of $\Delta_{G_n}$ and $\Delta_{G_n}/\{e_{1,2},e_{1,n}\}$, respectively, that is isomorphic to a $4$-cycle (see \cite[proof of Corollary 6.2]{LN}). If, by contradiction, such a subcomplex exists in $\Delta_{G_n}$, then it has to contain the vertex $e_{2,n}$ (respectively, $e_{n,2}$) since $e_{2,n}$ (respectively, $e_{n,2}$) is the only vertex lying in an edge with $e_{1,n}$ but not $e_{1,2}$ (respectively, vice versa). As $\{e_{2,n},e_{n,2}\}$ is not an edge of $\Delta_{G_n}$, such a subcomplex cannot exist. The same reasoning shows the corresponding statement for $\Delta_{G_n}/\{e_{1,2},e_{1,n}\}$ and $\{e_{2,1},e_{n,1}\}$. In particular, it follows that  $(\Delta_{G_n}/\{e_{1,2},e_{1,n}\})/\{e_{2,1},e_{n,1}\}$ is flag. We consider the simplicial map $\xi: (\Delta_{G_n}/\{e_{1,2},e_{1,n}\})/\{e_{2,1},e_{n,1}\}\to \Delta_{K_{2,n-2}}$, defined by $\xi(\pm e_{1,2})=\pm e_{1,n}$ and $\xi(e)=e$ for any other vertex of  $(\Delta_{G_n}/\{e_{1,2},e_{1,n}\})/\{e_{2,1},e_{n,1}\}$. Using \eqref{eq: E K}, \eqref{eq: E G} and the definition of edge contraction it is easy to check that $\xi$ induces a simplicial isomorphism between the $1$-skeleta of $(\Delta_{G_n}/\{e_{1,2},e_{1,n}\})/\{e_{2,1},e_{n,1}\}$ and $\Delta_{K_{2,n-2}}$, which shows the claim.
\end{proof}

We record here an explicit computation that will come in handy in the proof of \Cref{thm:lutznevoSEP} below.
\begin{example}\label{ex:contraction}
Consider the graphs $K_4$ and $G_4$ (with the labeling described previously) and order their edges so that $34 < 24 < 23 < 14 < 13 < 12$. Let $\Delta_{K_4}$ and $\Delta_{G_4}$ be the respective (flag) unimodular triangulations of $\partial\Pc_{K_4}$ and $\partial\Pc_{G_4}$ induced by this choice.

Consider the sequence of edge contractions
\[
        \Delta_{K_4}=:\Delta_0\to \Delta_1:=(\Delta_{K_4}/\{e_{1,4},e_{3,4}\})\to \Delta_2:=(\Delta_{K_4}/\{e_{1,4},e_{3,4}\})/\{e_{4,1},e_{4,3}\}.\]
        
One can check that $\Delta_1$ and $\Delta_2$ are flag spheres and, since $\Delta_0$ is $2$-dimensional, both $\lk_{\Delta_0}(\{e_{1,4},e_{3,4}\})$ and $\lk_{\Delta_1}(\{e_{4,1},e_{4,3}\})$ consist of two vertices. We claim that $\Delta_2$ is isomorphic to $\Delta_{G_4}$: since both complexes are flag, this can be verified by exhibiting a simplicial map between $\Delta_2$ and $\Delta_{G_4}$ which is an isomorphism on the $1$-skeleta. The map $\varphi\colon \Delta_2 \to \Delta_{G_4}$ defined by $\varphi(\pm e_{1,2}) = \pm e_{1,3}$, $\varphi(\pm e_{1,3}) = \pm e_{2,3}$, $\varphi(\pm e_{1,4}) = \pm e_{4,2}$, $\varphi(\pm e_{2,3}) = \pm e_{2,1}$, $\varphi(\pm e_{2,4}) = \pm e_{4,1}$ gives the desired result.
\end{example}

We can now state the main result of this section.

\begin{theorem} \label{thm:lutznevoSEP}
Let $G$ be a connected graph on $n \geq 5$ vertices. Then $\gamma_2(\Pc_{G})=0$ if and only if there exist a flag unimodular triangulation $\Delta_G$ of $\partial\Pc_G$ and a sequence of edge contractions $\Delta_{G}=:\Delta_0\to \Delta_1:=\Delta_0/F_1\to \Delta_2:=\Delta_1/F_2\to \cdots \to \Delta_{2k}:=\Delta_{2k-1}/F_k$ such that
\begin{itemize}
    \item[(i)] $\Delta_i$ is a flag sphere for every $0 \leq i \leq 2k$;
    \item[(ii)] $\Delta_{2k} \cong \Diamond_{n-2}$;
    \item[(iii)] $\lk_{\Delta_{i-1}}(F_i)\cong \Diamond_{n-4}$ for every $1 \leq i \leq 2k$.
\end{itemize}
Moreover, if the conditions above are met, for every $0 \leq i \leq k$ the complex $\Delta_{2i}$ is a unimodular triangulation of the boundary of some symmetric edge polytope.
\end{theorem}

\begin{proof}
    The validity of the ``if''-part has already been observed for general flag PL-spheres at the beginning of this section.
    
    For the other direction assume first that $G$ is $2$-connected. By \Cref{cor: gamma_2=0} we know that $\gamma_2(P_G)=0$ if and only if either $G\cong K_{2,n-2}$ or $G\cong G_n$. Iteratively applying \Cref{lem: two steps} and recalling that $(\Delta * \Gamma) / F = \Delta * (\Gamma / F)$ whenever $F$ is a face of $\Gamma$, we obtain the following chain of edge contractions and isomorphisms:
    \begin{align*}
        \Delta_{G_n}=\Delta_0&\to \Delta_1=(\Delta_{G_n}/\{e_{1,2},e_{1,n}\})\to \Delta_2=(\Delta_{G_n}/\{e_{1,2},e_{1,n}\})/\{e_{2,1},e_{n,1}\}\\
        &\overset{(ii)}{\cong}\Delta_{K_{2,n-2}}\overset{(i)}{\cong}=
        \langle e_{2,n},e_{n,2}\rangle *\Delta_{G_{n-1}}\\
        &\to\Delta_3=\langle e_{2,n},e_{n,2}\rangle * (\Delta_{G_{n-1}}/\{e_{1,2},e_{1,n-1}\})\\
        &\to\Delta_4=\langle e_{2,n},e_{n,2}\rangle * (\Delta_{G_{n-1}}/\{e_{1,2},e_{1,n-1}\})/\{e_{2,1},e_{n-1,1}\}\\
        &\overset{(ii)}{\cong}\langle e_{2,n},e_{n,2}\rangle*\Delta_{K_{2,n-3}}\overset{(i)}{\cong}\langle e_{2,n},e_{n,2}\rangle*\langle e_{2,n-1},e_{n-1,2}\rangle*\Delta_{G_{n-2}}\\
        &\hspace{100pt} \vdots\\
        & \Delta_{2n}\overset{(ii)}{\cong}\langle e_{2,n},e_{n,2}\rangle * \cdots * \langle e_{2,4},e_{4,2}\rangle * \Delta_{K_{2,1}} \cong \Diamond_{n-2},
    \end{align*}
    where the last isomorphism holds as  $\Delta_{K_{2,1}} \cong \Diamond_{1}$, and the $(n-3)$-fold suspension over $\Diamond_{1}$ is isomorphic to $\Diamond_{n-2}$. It follows from \Cref{lem: two steps} that all complexes in this sequence are flag. Moreover, the proof of \Cref{lem: two steps} (ii) shows that the links of the contracted edges need to satisfy the link condition, implying that all complexes in the sequence are triangulations of spheres (see \cite[Section 6]{LN} and \cite{NevoVanKampen}). Since the link of a simplex in a flag sphere is again a flag sphere, and $\gamma_1\geq 0$ for all flag spheres \cite{Gal, Meshulam}, a double application of \eqref{eq:gammaContraction} together with \Cref{cor: gamma_2=0} implies that $\gamma_1=0$ for every link of an edge that is contracted. As the only flag spheres with $\gamma_1=0$ are the boundaries of cross-polytopes (see \cite{Gal, Meshulam}), (iii) follows.

    The ``Moreover''-statement follows from the above sequence of contractions and the fact that adding a leaf to a graph corresponds to taking the suspension of the corresponding symmetric edge polytope.

        Finally, assume $G$ is not $2$-connected. Let $G=H_1\cup\cdots \cup H_k$ be its decomposition in the $2$-connected components $H_i$. Then $\Delta_{G}=\Delta_{H_1}*\cdots *\Delta_{H_k}$. \Cref{cor:gamma2=0general} implies that there exists at most one $i$ such that $H_i$ is not a single edge. If all $H_i$ are edges, then $\Pc_G$ is a cross-polytope and there is nothing to show in this case. Otherwise, without loss of generality, we can assume that $H_1$ is not an edge. It follows from the above proof and \Cref{ex:contraction} that $H_1$ admits edge contractions as required. As all $\Pc_{H_i}$ are line segments for any $2\leq i\leq k$ and since edge contractions and taking links commute with taking joins, the claim follows.
\end{proof}

\section{Symmetric edge polytopes for Erd\H{o}s-R\'enyi random graphs} \label{sec:probabilistic}
In this section, we consider symmetric edge polytopes for random graphs generated by the Erd\H{o}s-R\'enyi model. The ultimate goal is to prove \Cref{thm:B}. We try to keep this section self-contained and tailored for a reader without much knowledge of random graphs. However, we recommend \cite{AS-TheProbabilisticMethod,bollobas} and \cite{book-RandomGraphs} for more background on Erd\H{o}s-R\'enyi random graphs.

\subsection{Edges and cycles in Erd\H{o}s-R\'enyi graphs}
We write $G(n,p)$ for the Erd\H{o}s-R\'enyi probability model of random graphs on vertex set $[n]$, where edges are chosen independently with probability $p\in [0,1]$. Usually, $p:\mathbb{N}\to [0,1]$ is a function depending on $n$ that tends to $0$ at some rate as $n$ goes to infinity. For ease of notation we mostly just write $p$. We will say that a graph property $\mathcal{A}$, i.e., a family of graphs closed under isomorphism, holds \emph{asymptotically almost surely} (a.a.s. for short) or \emph{with high probability} if the probability that $G\in G(n,p)$ has property $\mathcal{A}$ tends to $1$ as $n$ goes to infinity, i.e.,
$$
\lim_{n\to\infty}\Pb(G\in \mathcal{A})=1 \qquad \text{for } G\in G(n,p).
$$
 In the following, given $G\in G(n,p)$, we will consider the symmetric edge polytope $\Pc_{G}$ of $G$. It follows from \eqref{eq:gamma_faces} and \Cref{lem: nonfaces of Delta<} that the $\gamma$-vector of $\Pc_{G}$ is independent of the vertex and edge labels of $G$ and hence, in particular, properties as $\gamma(\Pc_{G})$ being nonnegative or exhibiting a certain growth are graph properties as defined above. For the study of $\gamma_k(\Pc_G)$, the key idea is that its growth is governed by the number of cycles of length at most $2k$ in $G$. Therefore, we will take a detour through studying the number of cycles of length smaller than or equal to $2k$ in $G$ for $G\in G(n,p)$. Most of the results we need can be found somewhere in the literature (most often in more general form) and are probably well-known to the stochastics community. However, since we do not assume the typical reader of this article to be entirely familiar with this topic, we include proofs of most of the needed statements to keep this article as self-contained as possible.

We start by considering the number $X_E(G)$ of edges of $G\in G(n,p)$. This random variable is highly concentrated around its expectation. 

\begin{lemma}\label{lem:edgesER}
\begin{itemize}
    \item[(i)] $\Eb(X_E)=\binom{n}{2}p$,
    \item[(ii)] $\Var(X_E)=\binom{n}{2}(p-p^2)$,
    \item[(iii)] $\lim_{n\to\infty}\Pb(|X_E-\Eb(X_E)|\leq A\Eb(X_E))=1$ for any $A\in \R_{>0}$  and $p(n)=n^{-\beta}$ with $0\leq\beta\leq1$. 
\end{itemize}
\end{lemma}

\begin{proof}
(i) and (ii) follow from an easy computation. For (iii) Chebyshev's inequality implies
\begin{equation*}
    \Pb(|X_E-\Eb(X_E)|>A\Eb(X_E))\leq \frac{\Var(X_E)}{A^2\Eb(X_E)^2}
    =\frac{p-p^2}{A^2\binom{n}{2}p^2}\leq \frac{1}{Bn^{2-\beta}},
\end{equation*}
where $B\in \R$ is a positive constant. Since $\beta\leq 1$, the above expression tends to $0$ as $n$ goes to infinity, which shows the claim.
\end{proof}

For $G\in G(n,p)$ and $k \in \N$ we denote by $X_k(G)$ and $X(G)$ the number of $k$-cycles and cycles of any length in $G$, respectively. Moreover, based on the following lemma (see e.g.~\cite[Theorem 5.3]{book-RandomGraphs}) we will divide our study of $\gamma_\ell(\Pc_G)$, where $G\in G(n,p)$, into two cases.
\begin{lemma} \label{lem:nocycles}
Let $k\geq 3$ and let $G\in G(n,p)$. Then
$$\lim_{n\to \infty}\Pb(X_k>0)=\begin{cases}
0 \mbox{ if } \lim_{n\to \infty}np(n)= 0\\
1 \mbox{ if } \lim_{n\to \infty}np(n)= \infty.
\end{cases}.$$
\end{lemma}
In the following sections, we will distinguish between
\begin{itemize}
    \item the \emph{subcritical} regime, i.e., $\lim_{n\to \infty} n p(n)= 0$,
    \item the \emph{supercritical} regime, i.e., $\lim_{n\to \infty}n p(n)=\infty $.
\end{itemize}

\subsection{The subcritical regime}
We start by proving a strengthening of \Cref{lem:nocycles}. 
\begin{lemma}\label{lem:nocyclesatall}
Let $p(n)$ be such that $\lim_{n\to \infty}np(n)=0$. Then
$$
\lim_{n\to\infty}\Pb(X>0)=0.
$$
\end{lemma}

\begin{proof}
By Markov's inequality  we have
 \begin{equation}\label{eq:Markov}
     \Pb(X\geq 1)\leq \Eb(X).
\end{equation}
For an $\ell$-cycle $C$ in $K_n$, let $X_C$ be the indicator variable on $G(n,p)$ with $X_C(G)=1$ if $C\subseteq G$ and $X_C(G)=0$, otherwise. 
Then 
$\Eb(X_C)=\Pb(X_C=1)=p^\ell$ and since there are $\binom{n}{\ell}$ ways to choose $\ell$ vertices in $K_n$  out of which $\frac{(\ell-1)!}{2}$ different cycles can be built, we conclude
\begin{equation}\label{eq:expectation}
     \Eb(X_\ell)=\binom{n}{\ell}\frac{(\ell-1)!}{2}p^\ell.
 \end{equation}
 Using the linearity of expectation and \eqref{eq:Markov}, we further obtain
 \begin{equation*}
     \Pb(X>0)\leq \sum_{\ell=3}^n\Eb(X_\ell)
     =\sum_{\ell=3}^n\binom{n}{\ell}\frac{(\ell-1)!}{2}p^\ell
     \leq\frac{1}{2}\sum_{\ell=3}^n \frac{(pn)^\ell}{\ell}
     \leq  \frac{1}{2}\sum_{\ell=1}^\infty \frac{(pn)^\ell}{\ell}
 \end{equation*}
 As $\lim_{n\to \infty}np(n)=0$, we have $0<np(n)<1$ for $n$ large enough and hence the above series is convergent and equals  $-\ln(1-pn)$.
 Taking the limit we obtain
 $$
 \lim_{n\to\infty}  \Pb(X>0)\leq \lim_{n\to\infty}-\ln(1-pn)=-\ln(1)=0.
 $$
\end{proof}

The next theorem describes the behavior of the $\gamma$-vector in the subcritical regime.
\begin{theorem}\label{gamma:no cycles}
Let $p(n)$ be such that $\lim_{n\to \infty}np(n)=0$. Then
\begin{equation*}
   \lim_{n\to\infty} \Pb(\gamma_k=0\text{ for all } k\geq 1)=1 .
\end{equation*}
\end{theorem}

\begin{proof}
    \Cref{lem:nocyclesatall} implies that a.a.s. $G\in G(n,p)$ is a forest. As the $\gamma$-vector of the symmetric edge polytope of a forest equals $(1,0,\ldots,0)$, the claim follows.
\end{proof}

\begin{remark}
    We want to point out that \Cref{lem:nocyclesatall} implies that, in the subcritical regime, a.a.s. the symmetric edge polytope is a free sum of cross-polytopes, the number of summands being the number of components of the graph, and as such a cross-polytope itself.
\end{remark}

\subsection{The supercritical regime}\label{sect:supercritical}
We now consider the situation where $\lim_{n\to\infty}np(n)= \infty$. We start by computing the variance of the number $X_k$ of $k$-cycles.

\begin{proposition}\label{prop:Expectation and Variance}
Let $p(n)$ be such that $\lim_{n\to\infty}np(n)=\infty$. For $k\in \N$, $k\geq 3$ and $n$ large enough we have
$$\Var(X_k)\leq A\cdot \Eb(X_k)^2\cdot (np(n))^{-1},$$
where $A\in \R$ is a positive constant.
\end{proposition} 

We include a proof as a service to the reader.

\begin{proof}
We need to compute $\Var(X_k)=\Eb(X_k^2)-\Eb(X_k)^2$. 
As in the proof of \Cref{lem:nocyclesatall}, for a $k$-cycle $C\subseteq K_n$, we denote by $X_C$ the corresponding indicator variable. Moreover, we use $\mathcal{H}$ to denote the set of all $k$-cycles in $K_n$. By linearity of expectation, it follows that
\begin{equation}\label{eq:expectationSquared}
\Eb(X_k^2)=\sum_{C,C'\in\mathcal{H}}\Eb(X_C\cdot X_{C'})
=\sum_{C,C'\in\mathcal{H}}p^{2k-|E(C\cap C')|}\leq \sum_{C,C'\in\mathcal{H}}p^{2k-|V(C\cap C')|},
\end{equation}
where for the last inequality we use that $C\cap C'$ is a subgraph of a cycle and hence $|E(C\cap C')|\leq |V(C\cap C')|$. 
For $0\leq \ell\leq k$ we set $\mathcal{H}_\ell=\{(C,C')\in \cH^2~:~|V(C\cap C')|=\ell\}.$
If $(C,C')\in \cH_0$, the random variables $X_C$ and $X_{C'}$ are independent and we have
\begin{align*}
\sum_{(C,C')\in \cH_0}\Pb(C\cup C'\subseteq G)=&\sum_{(C,C')\in \cH_0}\Pb(C\subseteq G)\Pb( C'\subseteq G)\\
\leq&\left(\sum_{C\in \cH}\Pb(C\subseteq G)\right)\left(\sum_{C\in \cH}\Pb(C\subseteq G)\right)=\Eb(X_k)^2.
\end{align*}

For $\ell\geq 1$ a simple counting argument shows that
\begin{equation*}
    |\mathcal{H}_\ell|=\binom{n}{k}\frac{(k-1)!}{2}\binom{k}{\ell}\binom{n-k}{k-\ell}\frac{(k-1)!}{2}.
\end{equation*}
This together with \eqref{eq:expectationSquared} yields

\begingroup
\allowdisplaybreaks
\begin{align*}
    \Eb(X_k^2)&\leq \Eb(X_k)^2+\sum_{\ell=1}^k \binom{n}{k}\frac{(k-1)!}{2} \binom{k}{\ell}\binom{n-k}{k-\ell}\frac{(k-1)!}{2}p^{2k-\ell}\\
    &=\Eb(X_k)^2+\binom{n}{k}\frac{(k-1)!}{2}p^k\sum_{\ell=1}^k \binom{k}{\ell}\binom{n-k}{k-\ell}\frac{(k-1)!}{2}p^{k-\ell}\\
        &\leq \Eb(X_k)^2+\Eb(X_k)\sum_{\ell=1}^k A_{1,\ell}\binom{k}{\ell}n^{k-\ell}\frac{(k-1)!}{2}n^\ell p^k (np)^{-\ell}\\
    &\leq\Eb(X_k)^2+\Eb(X_k)\sum_{\ell=1}^k A_{2,\ell}\binom{k}{\ell}\binom{n}{k}\frac{(k-1)!}{2}p^k (np)^{-\ell}\\
    &=\Eb(X_k)^2+\Eb(X_k)\sum_{\ell=1}^k A_{2,\ell}\binom{k}{\ell}\Eb(X_k) (np)^{-\ell}\\
    &=\Eb(X_k)^2+\Eb(X_k)^2\sum_{\ell=1}^k A_{2,\ell}\binom{k}{\ell} (np)^{-\ell},
\end{align*}
\endgroup
where $A_{1,\ell},A_{2,\ell}\in \R$ are positive constants. 
If $\lim_{n\to\infty}np(n)=\infty$, then $(np)^{-\ell}\leq (np)^{-1}$ for large $n$ and hence
\begin{equation*}
\Eb(X_k^2)\leq \Eb(X_k)^2+\Eb(X_k)^2\sum_{\ell=1}^k A_{2,\ell}\binom{k}{\ell} (np)^{-1}=\Eb(X_k)^2+\Eb(X_k)^2\cdot A\cdot (np)^{-1} 
\end{equation*}
for large $n$, where $A=\sum_{\ell=1}^k A_{2,\ell}\binom{k}{\ell}$. The claim now follows from the definition of the variance.
\end{proof}

Using Chebyshev's inequality, \Cref{prop:Expectation and Variance} implies the following concentration inequalities for $X_k$.

\begin{corollary}\label{cor:concentrationCycles}
Let $p(n)$ be such that $\lim_{n\to\infty}np(n)=\infty$. For $k\in \N$, $k\geq 3$ and $A\in \R_{> 0}$ we have
\begin{equation*}
    \lim_{n\to \infty}\Pb(|X_k-\Eb(X_k)|\leq A\Eb(X_k))=1.
\end{equation*}
\end{corollary}
In the following, we assume that $p(n)=n^{-\beta}$ for some $0<\beta<1$. Using \Cref{cor:concentrationCycles} we show different concentration inequalities which are more convenient for our purposes. 
\begin{lemma}\label{lem:concentrationCycles}
Let  $0<\beta <1$, $p(n)=n^{-\beta}$, $\alpha=\min(\frac{1}{2},\frac{\beta}{2-\beta})$ and $k\in \N$, $k \geq 3$. Then  for $A\in \mathbb{R}_{>0}$ large enough we have that 
\begin{equation*}
  \lim_{n\to\infty} \Pb\left(\frac{1}{2}\Eb(X_{\ell})\leq X_{\ell}\leq A \Eb(X_E)^{\lceil \ell/2\rceil-\alpha} \mbox{ for all }3\leq\ell\leq k\right)=1.
  \end{equation*}
\end{lemma}

\begin{proof}
We note that it suffices to show the existence of some constant $A$ satisfying the claimed statement, since then every $A'\geq A$ satisfies it as well.

By \eqref{eq:expectation} and \Cref{lem:edgesER}, for $n$ large enough, it holds that 
\begin{align*}
    \Eb(X_{2\ell})\le &A_1\left(\binom{n}{2}p\right)^\ell p^\ell=A_1\Eb(X_E)^\ell (n^{2-\beta})^{\frac{-\beta \ell}{2-\beta}}\\
   \le &A_2\Eb(X_E)^\ell \left(\binom{n}{2}p\right)^{\frac{-\beta \ell}{2-\beta}}
    \leq A_2\Eb(X_E)^\ell \cdot \Eb(X_E)^{\frac{-\beta }{2-\beta}}
    \leq A_2 \Eb(X_E)^{\ell-\alpha},
\end{align*}
where for the last two inequalities we use that $\Eb(X_E)\geq 1$ for $n$ large enough and $A_1,A_2\in \R$ are positive constants. This yields for $A\in\R_{>0}$ and $n$ large enough
\begin{equation*}
    \Pb( X_{2\ell}<  (1+A)\Eb(X_{2\ell}))\leq \Pb( X_{2\ell}< (1+A)A_2\Eb(X_E)^{\ell-\alpha}).
    \end{equation*}
Setting $A_3=(1+A)A_2$ we infer from \Cref{cor:concentrationCycles} that $\lim_{n\to\infty}\Pb( X_{2\ell}< A_3\Eb(X_E)^{\ell-\alpha})=1$. 
Since, again by \Cref{cor:concentrationCycles}, 
\begin{equation*}
    \lim_{n\to\infty}\Pb(X_{2\ell}< \frac{1}{2}\Eb(X_{2\ell}))\leq \lim_{n\to\infty}\Pb(|X_{2\ell}-\Eb(X_{2\ell})|> \frac{1}{2}\Eb(X_{2\ell}))=0,
\end{equation*}
we obtain
\begin{equation}\label{eq:evenCycles}
    \lim_{n\to\infty}\Pb(X_{2\ell}< \frac{1}{2}\Eb(X_{2\ell})\mbox{ or } X_{2\ell}> A_3\Eb(X_E)^{\ell-\alpha})=0.
\end{equation}
For odd cycles, a similar computation as for even cycles shows that for $n$ large enough
\begin{equation*}
    \Eb(X_{2\ell-1})\leq A_4\cdot \Eb(X_E)^{\ell-\frac{1}{2}}\cdot n^{-\beta\ell+\frac{1}{2}\beta}
    \leq A_4\cdot \Eb(X_E)^{\ell-\frac{1}{2}}\leq A_4\cdot \Eb(X_E)^{\ell-\alpha},
\end{equation*}
where $A_4\in \R$ is a positive constant and for the last inequality we use that $\Eb(X_E)\geq 1$ for $n$ large enough.
Almost the same argument as for even cycles implies that 
\begin{equation}\label{eq:oddcycles}
    \lim_{n\to\infty}\Pb(X_{2\ell-1}<\frac{1}{2}\Eb(X_{2\ell-1})\mbox{ or } X_{2\ell-1}>A_4 \Eb(X_E)^{\ell-\alpha})=0.
\end{equation}
Combining \eqref{eq:evenCycles} and \eqref{eq:oddcycles}  we finally get
\begin{align*}
    &\lim_{n\to\infty} \Pb\left(\frac{1}{2}\Eb(X_{\ell})\leq X_{\ell}\leq A\Eb(X_E)^{\lceil \ell/2\rceil-\alpha} \mbox{ for all }3\leq\ell\leq k\right)\\
    &\geq 1-\sum_{\ell=3}^k\lim_{n\to\infty}\Pb(X_{\ell}<\frac{1}{2}\Eb(X_{\ell})\mbox{ or } X_{\ell}> A\Eb(X_E)^{\lceil\frac{\ell}{2}\rceil-\alpha})=1,
\end{align*}
    where $A$ is taken as the maximal constant appearing in \eqref{eq:evenCycles} and \eqref{eq:oddcycles} for $3\leq \ell\leq k$.
\end{proof}

To get information about the $\gamma$-vector of the symmetric edge polytope of a random graph $G\in G(n,p)$ we want to use \Cref{lem: nonfaces of Delta<}. The first part of our strategy consists in turning the concentration inequalities of \Cref{lem:concentrationCycles} into concentration inequalities for the number of non-faces and faces of bounded cardinality. In a second step, we use the latter to infer concentration inequalities for the $\gamma$-vector up to a fixed entry. We now make this idea more precise. 
Given a graph $G$ on $n$ vertices, we let $\Delta_G$ be a unimodular triangulation of $\partial \Pc_G$ as described in \Cref{lem: nonfaces of Delta<}. Since $\Pc_G$ is reflexive and $\Delta_G$ is unimodular, we have $h_j^\ast(\Pc_G)=h_j(\Delta_G)$ for every $j$. Using the symmetry of $h(\Delta_G)$ and the definition of the $\gamma$-vector we further know that
\[
\sum_{i=0}^{\lfloor \frac{\dim\Pc_G}{2}\rfloor} \gamma_i(\Pc_G) t^i(t+1)^{\dim\Pc_G-2i}=
\sum_{j=0}^{\dim\Pc_G} h^\ast_j(\Pc_G) t^{\dim\Pc_G-j}.
\]
The usual relation between the $f$- and $h$-vector of a simplicial complex together with the substitution of $t$ by $t+1$ implies 

\begin{equation*}
\sum_{j=0}^{\dim\Pc_G} f_{j-1}(\Delta_G) t^{\dim\Pc_G-j} = \sum_{i=0}^{\lfloor \frac{\dim\Pc_G}{2}\rfloor} \gamma_i(\Pc_G)(t+1)^i(t+2)^{\dim\Pc_G-2i}.
\end{equation*}
In particular, evaluating the coefficient of $t^{\dim\Pc_G-k}$ (and truncating the right hand side at $i=k$) yields that
\begin{equation} \label{eq:gamma_faces}
\gamma_k(\Pc_G) = f_{k-1}(\Delta_G) - [t^{\dim\Pc_G-k}] \sum_{i=0}^{k-1} \gamma_i(\Pc_G)(t+1)^i(t+2)^{\dim\Pc_G-2i}.
\end{equation}
In order to be able to make sense out of \eqref{eq:gamma_faces}, we need to know what the dimension of $\Pc_G$ is for ``most'' Erd\H{o}s-R\'enyi graphs $G\in G(n,p)$ in the supercritical regime. Denoting by $X_{\dim\Pc}$ the corresponding random variable, we have:

\begin{lemma}\label{dim:Sep}
Let $0<\beta < 1$, $p(n)=n^{-\beta}$. Then, for $G\in G(n,p)$, we have
\begin{equation*}
    \lim_{n\to \infty}\Pb(X_{\dim\Pc}=n-1)=1.
\end{equation*}
\end{lemma}

\begin{proof}
Since for any graph $G$ on $n$ vertices $\dim\Pc_G=n-1$ if and only if $G$ is connected, the result follows e.g.~from \cite[Theorem 4.1]{book-RandomGraphs} using the fact that $n^{-\beta}$ grows faster than $\frac{\log(n)}{n}$ for any $0<\beta<1$.
\end{proof}

In order to use \eqref{eq:gamma_faces} to show concentration inequalities for the $\gamma$-vector of $\Pc_G$, we need to study the random variables $f_{k-1}$ or equivalently the number of non-faces of $\Delta_G$. For $G\in G(n,p)$ we denote by $n_{k-1}(G)$ the number of $(k-1)$-dimensional non-faces of $\Delta_G$ that do not contain antipodal vertices. Note that $n_1(G)$ equals the number of bad pairs of $G$ as in \Cref{sec:nonnegativity}. 

\begin{theorem}\label{thm:concentration Nonfaces}
Let $0<\beta < 1$, $p(n)=n^{-\beta}$, $\alpha=\min(\frac{1}{2},\frac{\beta}{2-\beta})$ and $k\in \N$, $k \geq 1$. Then for $B\in \R_{>0}$ large enough  we have
\begin{equation*}
    \lim_{n\to\infty}\Pb(n_{\ell-1}\leq B\Eb(X_E)^{\ell-\alpha} \mbox{ for all }2\leq \ell\leq k+1)=1.
    \end{equation*}
\end{theorem}

\begin{proof}
As in the proof of \Cref{lem:concentrationCycles}, it suffices to show the existence of some constant $B$ satisfying the claimed statement.

Let $G\in G(n,p)$ and let $2\leq \ell\leq k+1$. On the one hand, any $(\ell-1)$-non-face of $\Delta_G$ contains a minimal (not necessarily unique) $r$-non-face of $\Delta_G$ for some $1\leq r\leq \ell-1$. On the other hand, any such minimal $r$-non-face of $\Delta_G$ can be extended to an $(\ell-1)$-non-face of $\Delta_G$ by adding $\ell-1-r$ non-antipodal vertices to it, for which there are $\binom{X_E(G)-(r+1)}{\ell-1-r}\cdot 2^{\ell-1-r}$ possibilities. Hence, denoting by $N_r(G)$ the number of minimal $r$-non-faces of $\Delta_G$, we conclude
\begin{equation*}
    n_{\ell-1}(G)\leq \sum_{r=1}^{\ell-1}\binom{X_E(G)-(r+1)}{\ell-1-r}\cdot 2^{\ell-1-r}\cdot N_r(G)
    \le\sum_{r=2}^\ell B_{r,\ell} X_E(G)^{\ell-r}N_{r-1}(G),
\end{equation*}
where $B_{r,\ell}\in\R$ are positive constants.
Let $A_1\in \R$ such that \Cref{lem:concentrationCycles} holds. As, by \Cref{lem: nonfaces of Delta<}, for $G\in G(n,p)$ we have 
$N_{r-1}(G)\leq 2\cdot\binom{2r-1}{r}(X_{2r}(G)+X_{2r-1}(G))$, it follows that 
\begin{equation*}
    \Pb\left(N_{r-1}\leq 4\cdot \binom{2r-1}{r}\cdot A_1 \cdot \Eb(X_E)^{r-\alpha}\right)\geq \Pb\left(X_{2r}\leq A_1\Eb(X_E)^{r-\alpha} \mbox{ and } X_{2r-1}\leq A_1\Eb(X_E)^{r-\alpha}\right)
    \end{equation*}
    and by the choice of $A_1$ we have $\lim_{n\to\infty} \Pb(N_{r-1}\leq 4\cdot \binom{2r-1}{r}\cdot A_1 \cdot \Eb(X_E)^{r-\alpha})=1$. As, by \Cref{lem:edgesER} (iii), we also have $\lim_{n\to \infty}\Pb(X_E\leq (1+A_2)\Eb(X_E))=1$ for any $A_2\in \R_{>0}$, we conclude that a.a.s. it holds that
\begin{align*}
    n_{\ell-1}\leq &\sum_{r=2}^\ell B_{r,\ell} (1+A_2)^{\ell-r}\Eb(X_E)^{\ell-r}\cdot4\cdot \binom{2r-1}{r}\cdot A_1 \cdot \Eb(X_E)^{r-\alpha}\\
    =&\left(\sum_{r=2}^\ell B_{r,\ell} (1+A_2)^{\ell-r}\cdot4\cdot \binom{2r-1}{r}\cdot A_1 \right) \Eb(X_E)^{\ell-\alpha}=B\cdot\Eb(X_E)^{\ell-\alpha}
\end{align*}
with $B=\sum_{r=2}^\ell B_{r,\ell} (1+A_2)^{\ell-r}\cdot4\cdot \binom{2r-1}{r}\cdot A_1$. The claim follows. 
\end{proof}

For $G\in G(n,p)$ we denote by $f_{k-1}(G)$ the number of $(k-1)$-faces of $\Delta_G$. 
From \Cref{thm:concentration Nonfaces} we can deduce concentration inequalities for these random variables.

\begin{theorem}\label{thm:ConcenctrationInequalitiesFaces}
Let $0<\beta < 1$, $p(n)=n^{-\beta}$, $\alpha=\min(\frac{1}{2},\frac{\beta}{2-\beta})$ and $k\in \N$, $k \geq 1$. Then, for $\epsilon>0$ and $B\in\R_{>0}$ large enough, we have
\begin{equation*}
    \lim_{n\to\infty}\Pb\left(2^{\ell-\epsilon}\binom{\Eb(X_E)}{\ell}-B\Eb(X_E)^{\ell-\alpha}\leq  f_{\ell-1}\leq 2^{\ell+\epsilon}\binom{\Eb(X_E)}{\ell} \mbox{ for all } 1\leq \ell\leq k\right)=1.
    \end{equation*}
In particular,
\begin{equation}\label{eq:FacesAsymptotics}
\lim_{n\to\infty}\Pb\left(f_{\ell-1}\in \Theta(n^{(2-\beta)\ell}) \mbox{ for all } 1\leq \ell\leq k\right)=1.
\end{equation}
\end{theorem}

\begin{proof}
The statement trivially holds for $k=1$ since $f_0=2X_E$. Let $k\geq 2$ and $1 \leq \ell \leq k$.
For $G\in G(n,p)$ we have
\begin{equation*}
    f_{\ell-1}(G)= 2^\ell\binom{X_E(G)}{\ell}-n_{\ell-1}(G).
\end{equation*} 
Let $0<A_1<1$ be such that $\binom{(1-A_1)\Eb(X_E)}{\ell}=2^{-\epsilon}\binom{\Eb(X_E)}{\ell}$.  

It follows from \Cref{thm:concentration Nonfaces} and \Cref{lem:edgesER} (iii) that for large enough $B\in \R_{>0}$
\begin{equation}\label{eq:lower bound}
  \lim_{n\to\infty}\Pb\left(  f_{\ell-1}\geq 2^{\ell-\epsilon}\binom{\Eb(X_E)}{\ell}-B\Eb(X_E)^{\ell-\alpha}\right)=1.
\end{equation}
Finally, let $A_2>0$ be such that $\binom{(1+A_2)\Eb(X_E)}{\ell}=2^{\epsilon}\binom{\Eb(X_E)}{\ell}$. As for $G\in G(n,p)$ the triangulation $\Delta_G$ is a subcomplex of a cross-polytope of dimension $X_E(G)$, we can bound $f_{\ell-1}(G)$ from above by $2^\ell\binom{X_E(G)}{\ell}$. Using \Cref{lem:edgesER} (iii) we conclude that 
\begin{equation}\label{eq:upper bound}
   \lim_{n\to\infty} \Pb\left(f_{\ell-1}\leq 2^{\ell+\epsilon}\binom{\Eb(X_E)}{\ell}\right)=1.
\end{equation}
Combining \eqref{eq:lower bound}  and \eqref{eq:upper bound} for any $1\leq \ell\leq k$ finishes the proof of the first statement. 

For the ``In particular''-part it suffices to note that, since $\Eb(X_E)\in \Theta(n^{2-\beta})$, the upper and lower bounds for $f_{\ell-1}$ both lie in $\Theta(n^{(2-\beta)\ell})$.
\end{proof}

We are now ready to state the main result of this subsection.
\begin{theorem}\label{thm:gammaconcentration1}
Let $0<\beta < 1$, $p(n)=n^{-\beta}$ and $k\in \N$. Then
\begin{equation*}
    \lim_{n\to\infty}\Pb(\gamma_\ell\in \Theta(n^{(2-\beta)\ell}) \mbox{ for all } 0\leq \ell\leq k)=1.
\end{equation*}
\end{theorem}

\begin{proof}
We show the statement by induction on $k$. Since $\gamma_0=1$ is constant, the statement holds for $k=0$. 

Now assume $k\geq 1$. Since by the induction hypothesis we have
\begin{equation}\label{eq:nonnegativity}
    \lim_{n\to\infty }\Pb(\gamma_\ell\geq 0 \mbox{ for all } 0\leq \ell\leq k-1)=1,
\end{equation}
it follows from \eqref{eq:gamma_faces} that
\begin{equation}\label{eq:upperBound}
\lim_{n\to\infty}\Pb(\gamma_\ell\leq f_{\ell-1}\mbox{ for all } 0\leq \ell\leq k)=1.
\end{equation}
Thus, we have an upper bound for $\gamma_\ell$, which a.a.s. lies in $\Theta(n^{(2-\beta)\ell})$ by \eqref{eq:FacesAsymptotics}. Combining this upper bound with a more detailed analysis of \eqref{eq:gamma_faces} will enable us to prove that $\gamma_\ell$ can be bounded asymptotically always surely by a lower bound that also lies in $\Theta(n^{(2-\beta)\ell})$. \Cref{dim:Sep} implies that $X_{\dim\Pc}=n-1$ a.a.s.; hence, by \eqref{eq:upperBound} and \eqref{eq:gamma_faces} we have a.a.s. 
\begin{align}\label{eq:gammaThroughF}
   \gamma_\ell\geq& f_{\ell-1}-[t^{n-1-\ell}]\sum_{i=0}^{\ell-1} f_{i-1}(t+1)^i(t+2)^{n-1-2i}\\ \notag
   =& f_{\ell-1}-\sum_{i=0}^{\ell-1} f_{i-1}\left(\sum_{j=n-1-\ell-i}^{n-1-\ell}2^{n-1-2i-j}\binom{i}{n-1-\ell-j}\binom{n-1-2i}{j}\right). 
\end{align}
Using \Cref{thm:ConcenctrationInequalitiesFaces} we conclude that for large enough $B\in \R_{>0}$  it holds a.a.s. that
\begin{equation*}
   \gamma_\ell\geq 2^{\ell-\epsilon}\binom{\Eb(X_E)}{\ell}-B\Eb(X_E)^{\ell-\alpha}-\sum_{i=0}^{\ell-1} 2^{i+\epsilon}\binom{\Eb(X_E)}{i}\left(\sum_{j=n-1-\ell-i}^{n-1-\ell}2^{n-1-2i-j}\binom{i}{n-1-\ell-j}\binom{n-1-2i}{j}\right).
   \end{equation*}
  Since for $n\geq 2\ell+1$ one has $n-1-\ell-i\geq \frac{n-1-2i}{2}$, the expression $\binom{n-1-2i}{j}$ in the last sum is maximal for $j=n-1-\ell-i$. As also $2^{n-1-2i-j}$ is maximal in this case and $\binom{i}{n-1-\ell-j}\leq \binom{\ell-1}{\lfloor(\ell-1)/2\rfloor}$ for $0\leq i\leq \ell-1$ and any $j$, it follows that a.a.s. 
   \begin{align*}
       \gamma_\ell&\geq 2^{\ell-\epsilon}\binom{\Eb(X_E)}{\ell}-B\Eb(X_E)^{\ell-\alpha}- \binom{\ell-1}{\lfloor(\ell-1)/2\rfloor}\sum_{i=0}^{\ell-1} 2^{i+\epsilon}\binom{\Eb(X_E)}{i}(i+1)2^{\ell-i}\binom{n-1-2i}{n-1-\ell-i}\\
       &=2^{\ell-\epsilon}\binom{\Eb(X_E)}{\ell}-B\Eb(X_E)^{\ell-\alpha}- \binom{\ell-1}{\lfloor(\ell-1)/2\rfloor}\sum_{i=0}^{\ell-1} 2^{\ell+\epsilon}\binom{\Eb(X_E)}{i}(i+1)\binom{n-1-2i}{\ell-i}.
   \end{align*}
Analysing the expressions in the last equation, we see that
\begin{equation*}
    2^{\ell-\epsilon}\binom{\Eb(X_E)}{\ell}-B\Eb(X_E)^{\ell-\alpha}\in \Theta(n^{(2-\beta)\ell})
    \end{equation*}
    and
    \begin{equation*}
        2^{\ell+\epsilon}\binom{\Eb(X_E)}{i}(i+1)\binom{n-1-2i}{\ell-i}\in \Theta(n^{(2-\beta)i}\cdot n^{\ell-i}) =\Theta(n^{i+\ell-i\beta}).
    \end{equation*}
As $(2-\beta)\ell>i+\ell-i\beta $ for $\ell>i$, this implies
\begin{equation*}
    2^{\ell-\epsilon}\binom{\Eb(X_E)}{\ell}-B\Eb(X_E)^{\ell-\alpha}-\sum_{i=0}^{\ell-1} 2^{\ell+\epsilon}\binom{\Eb(X_E)}{i}(i+1)\binom{n-1-2i}{\ell-i}\in \Theta(n^{(2-\beta)\ell}).
\end{equation*}
As a consequence we have found that $\gamma_\ell$ can a.a.s. be bounded from below by an expression in $\Theta(n^{(2-\beta)\ell})$. Combining this with the previously shown upper bound completes the proof.
\end{proof}

\begin{remark}
It is natural to ask if the results for $\gamma_k$ we obtained in the subcritical and the supercritical regime (Theorems \ref{gamma:no cycles} and \ref{thm:gammaconcentration1}) can be extended to the critical regime, i.e., $p(n)=\frac{c}{n}$ for a constant $c>0$. Indeed, using that in this regime $X_\ell$ converges in distribution to a Poisson distribution with mean and variance $\frac{c^\ell}{2\ell}$ (see e.g., \cite[Theorem 10.1.1]{AS-TheProbabilisticMethod}), one can show that $X_\ell$ is highly concentrated around its mean. More precisely, 
 \begin{equation*}
   \lim_{n\to\infty}\Pb\left(|X_\ell-\Eb(X_\ell)|\leq \omega(n) \right)=1
\end{equation*}
for any arbitrarily slowly  increasing function $\omega:\mathbb{N}\to \mathbb{R}$. By similar arguments as in the proof of \Cref{thm:concentration Nonfaces}, this gives rise to the following concentration inequality for the non-faces:
\begin{equation}
    \lim_{n\to\infty}\Pb(n_{\ell-1}\leq n^{\kappa}\Eb(n_{\ell-1}))=1.
\end{equation}
By the same method as in the proof of \Cref{thm:ConcenctrationInequalitiesFaces}, one can show that 
\begin{equation}\label{eq:criticalFaces}
     \lim_{n\to\infty}\Pb(f_{\ell-1}\in \Theta(n^\ell))=1.
\end{equation}
Unfortunately, the arguments from the proof of  \Cref{thm:gammaconcentration1} only allow us to bound the double sum in the second row of \eqref{eq:gammaThroughF} by an expression in $\Theta(n^\ell)$. Hence, in order to be able to turn \eqref{eq:criticalFaces} into concentration inequalities for $\gamma_\ell$, different arguments or at least a  more refined analysis including the leading coefficients would be needed. It is reasonable to believe that, analogously to the variety of behaviors of the largest component of an Erd\H{o}s-R\'enyi graph (see e.g.~\cite[Chapter 11]{AS-TheProbabilisticMethod}), one would also get different behaviors for $\gamma_\ell$ depending on whether $c<1$, $c=1$ or $c>1$. We leave this as an open problem.
\end{remark}

\bibliographystyle{alpha} 
\bibliography{bibliography}
\end{document}